\documentclass[11pt]{article}
\usepackage{amsmath,amssymb,amsthm,amscd}

\newtheorem{theorem}{Theorem}[section]
\newtheorem{lemma}[theorem]{Lemma}
\newtheorem{proposition}[theorem]{Proposition}
\newtheorem{corollary}[theorem]{Corollary}

\theoremstyle{plain}

\theoremstyle{definition}
\newtheorem{definition}[theorem]{Definition}

\numberwithin{equation}{section}

\renewcommand{\labelenumi}{\textup{(\theenumi)}}

\title{One-sided topological conjugacy of normal subshifts and
 gauge actions on the associated $C^*$-algebras}
\author{Kengo Matsumoto \\
Department of Mathematics \\
Joetsu University of Education \\
Joetsu, 943-8512, Japan
}

\begin{document}
\maketitle

\date{}

%\email{kengo{@@}juen.ac.jp}
\begin{abstract}
The class of normal subshifts includes irreducible nontrivial topological Markov shifts,
 irreducible nontrivial sofic shifts, synchronized systems,
Dyck shifts, $\beta$-shifts, substitution minimal shifts, and so on.
We will characterize one-sided topological conjugacy classes
of normal subshifts in terms of the associated $C^*$-algebras 
and its gauge actions with potentials.
\end{abstract}



%Mathematics Subject Classification 2000:

%%%%%%%%%%%%%%%%%%%%%%%%%%%%%%%%%%%%%%%%%%%%%%%%%%%%%%                        

\newcommand{\R}{\mathbb{R}}
\newcommand{\T}{\mathbb{T}}
\newcommand{\Z}{\mathbb{Z}}
\newcommand{\N}{\mathbb{N}}
\newcommand{\Zp}{{\mathbb{Z}}_+}

\def\LLL{{ {\frak L}^{\lambda(\Lambda)} }}
\def\LLTL{{ {\frak L}^{\lambda(\widetilde{\Lambda})} }}

\def\OA{{ {\mathcal{O}}_A }}
\def\OB{{ {\mathcal{O}}_B }}
\def\DA{{ {\mathcal{D}}_A }}
\def\DB{{ {\mathcal{D}}_B }}

\def\OL{{ {\mathcal{O}}_{\frak L} }}
\def\DL{{ {\mathcal{D}}_{\frak L}}}
\def\AL{{ {\mathcal{A}}_{\frak L}}}
\def\OLmin{{ {\mathcal{O}}_{{\Lambda}^{\operatorname{min}}} }}
\def\LLmin{{ {{\frak L}_\Lambda^{\operatorname{min}}} }}
\def\OLamonemin{{ {\mathcal{O}}_{{\Lambda_1}^{\operatorname{min}}} }}
\def\LLamonemin{{ {{\frak L}_{\Lambda_1}^{\operatorname{min}}} }}
\def\OLamtwomin{{ {\mathcal{O}}_{{\Lambda_2}^{\operatorname{min}}} }}
\def\LLamtwomin{{ {{\frak L}_{\Lambda_2}^{\operatorname{min}}} }}
\def\LLamptwomin{{ {{\frak L}_{\Lambda'_2}^{\operatorname{min}}} }}

\def\DLLamtwomin{{ {\mathcal{D}}_{{\frak L}_{\Lambda_2}^{\operatorname{min}} } }}
\def\DLamtwo{{ {\mathcal{D}}_{{\Lambda_2} } }}
\def\DLamtwomin{{ {\mathcal{D}}_{{\Lambda_2}^{\operatorname{min}}} }}

\def\OLamptwomin{{ {\mathcal{O}}_{{\Lambda'_2}^{\operatorname{min}}} }}
\def\DLLamptwomin{{ {\mathcal{D}}_{{\frak L}_{\Lambda'_2}^{\operatorname{min}}} }}
\def\DLamptwo{{ {\mathcal{D}}_{{\Lambda'_2}  } }}
\def\DLamptwomin{{ {\mathcal{D}}_{{\Lambda'_2}^{\operatorname{min}}} }}
\def\ALLamptwomin{{ {\mathcal{A}}_{{\frak L}_{\Lambda'_2}^{\operatorname{min}}} }}

\def\OhatLamtwomin{{ \widehat{\mathcal{O}}_{{\Lambda'_2}^{\operatorname{min}}} }}
\def\DhatLamtwo{{ \widehat{\mathcal{D}}_{{\Lambda'_2}} }}
\def\DhatLamtwomin{{ \widehat{\mathcal{D}}_{{\Lambda'_2}^{\operatorname{min}}} }}
\def\DhatLLamtwomin{{ 
\widehat{\mathcal{D}}_{{\frak L}_{\Lambda'_2}^{\operatorname{min}}} 
}}

\def\ALLamptwomin{{ {\mathcal{A}}_{{\frak L}_{\Lambda'_2}^{\operatorname{min}}} }}
\def\AhatLLamtwomin{{ 
\widehat{\mathcal{A}}_{{\frak L}_{\Lambda'_2}^{\operatorname{min}}} 
}}

\def\rhoLamtwo{ \rho^{{\Lambda_2}} }
\def\rhoLamptwo{ \rho^{{\Lambda'_2}} }
\def\rhohatLamtwo{{ {\hat{\rho}}^{{\Lambda'_2}} }}
\def\rhohatLamtwomin{{ {\hat{\rho}}^{{\Lambda'_2}^{\operatorname{min}}} }}
\def\rhoLamtwomin{ {\rho}^{{\Lambda_2}^{\operatorname{min}}} }
\def\rhoLamptwomin{ {\rho}^{{\Lambda'_2}^{\operatorname{min}}} }
    
\def\FLamtwomin{{ {\mathcal{F}}_{{\Lambda_2}^{\operatorname{min}}} }}
\def\FhatLamtwomin{{ \widehat{\mathcal{F}}_{{\Lambda'_2}^{\operatorname{min}}} }}
\def\FLLamptwomin{{ {\mathcal{F}}_{{\frak L}_{\Lambda'_2}^{\operatorname{min}}} }}

\def\FL{{{\mathcal{F}}_{\frak L}}}
\def\FLmin{{ {\mathcal{F}}_{{\Lambda}^{\operatorname{min}}} }}

\def\DLmin{{ {\mathcal{D}}_{{\frak L}_\Lambda^{\operatorname{min}}} }}

\newcommand{\K}{\mathcal{K}}
\newcommand{\C}{\mathcal{C}}
\def\DLam{{ {\mathcal{D}}_{\Lambda}}}

\def\L{{\frak L}}

\def\OLL{{ {\cal O}_{\lambda(\Lambda)}  }}
\def\OLTL{{ {\cal O}_{\lambda(\widetilde{\Lambda})}  }}
\def\ALL{{ {\cal A}_{\lambda(\Lambda)}  }}
\def\ALTL{{ {\cal A}_{\lambda(\widetilde{\Lambda})}  }}
\def\DLL{{ {\cal D}_{\lambda(\Lambda)}  }}
\def\DLTL{{ {\cal D}_{\lambda(\widetilde{\Lambda})}  }}

\def\FKL{{ {\cal F}_k^{l} }}
\def\A{{ {\mathcal{A}} }}
\def\D{{ {\mathcal{D}} }}
\def\F{{ {\mathcal{F}} }}
\def\Ext{{{\operatorname{Ext}}}}
\def\Im{{{\operatorname{Im}}}}
\def\Aut{{{\operatorname{Aut}}}}
\def\Ad{{{\operatorname{Ad}}}}
\def\Hom{{{\operatorname{Hom}}}}
\def\Ker{{{\operatorname{Ker}}}}
\def\dim{{{\operatorname{dim}}}}
\def\min{{{\operatorname{min}}}}
\def\id{{{\operatorname{id}}}}
\def\OLF{{{\cal O}_{{\frak L}^{Ch(D_F)}}}}
\def\OLN{{{\cal O}_{{\frak L}^{Ch(D_N)}}}}
\def\OLA{{{\cal O}_{{\frak L}^{Ch(D_A)}}}}
\def\LCHDA{{{{\frak L}^{Ch(D_A)}}}}
\def\LCHDF{{{{\frak L}^{Ch(D_F)}}}}
\def\LCHLA{{{{\frak L}^{Ch(\Lambda_A)}}}}
\def\LWA{{{{\frak L}^{W(\Lambda_A)}}}}

%%%%%%%%%%%%%%%%%%%%%%%%%%%%%%%%%%%%%%% 

\bigskip

\section{Introduction}
Let $\Sigma$ be a finite set $\{ 1,2,\dots, N\}, N\ge 2$.
We denote by $\Sigma^\N $ the set of right infinite sequences of $\Sigma$ 
with its infinite product topology, so that $\Sigma^\N $  
is homeomorphic to a Cantor discontinuum.
Let $\sigma:\Sigma^\N \longrightarrow \Sigma^\N$ be the shift defined by
$\sigma((x_i)_{i \in \N}) = (x_{i+1})_{i\in \N}$ for $(x_i)_{i\in \N} \in \Sigma^\N$.
It gives rise to a continuous surjection on $\Sigma^\N$.
The topological dynamical system $(\Sigma^\N, \sigma)$ 
is called the one-sided full shift of order $N$.
The two-sided full shift $(\Sigma^\Z, \bar{\sigma})$ 
is similarly defined with the shift homeomorphism
$\bar{\sigma}((x_i)_{i \in \Z}) = (x_{i+1})_{i\in \Z}$ 
for $(x_i)_{i\in \Z} \in \Sigma^\Z$.
A closed $\bar{\sigma}$-invariant subset $\Lambda$ of $\Sigma^\Z$, 
that is $\bar{\sigma}(\Lambda) = \Lambda$, 
yields a topological dynamical system
$(\Lambda, \bar{\sigma}_{\Lambda})$ by the homeomorphism $\bar{\sigma}_{\Lambda}$
restricting $\bar{\sigma}$ to $\Lambda$.
The topological dynamical system  $(\Lambda, \bar{\sigma}_{\Lambda})$
is called a two-sided subshift, or a subshift for brevity.
Its one-sided subshift $(X_\Lambda,\sigma_\Lambda)$ is defined
by the shift space
$
X_\Lambda =\{(a_n)_{n\in \N} \in \Sigma^\N \mid (a_n)_{n \in \Z} \in \Lambda \}
$
with the shift $\sigma_\Lambda = \sigma|_{X_\Lambda}$.  
Two-sided subshift $(\Lambda, \bar{\sigma}_{\Lambda})$ and its one-sided subshift
 $(X_\Lambda,\sigma_\Lambda)$ are often written $\Lambda$ and $X_\Lambda$, 
 respectively for brevity.
The dynamical systems of subshifts are also  called symbolic dynamical systems.
The symbolic dynamical systems form 
a bilding block of the theory of topological dynamical systems.
For general theory of symbolic dynamical systems, 
see the text books \cite{Kitchens}, \cite{Kurka} \cite{LM}.

In symbolic dynamical systems, topological Markov shifts, 
often called shifts of finite type,
form a basic class of subshifts.
For an $N\times N$ matrix $A =[A(i,j)]_{i,j=1,\dots,N}$ with entries in $\{0,1\}$,
the two-sided topological Markov shift $(\Lambda_A,\bar{\sigma}_A)$ 
is defined as a subshift by setting
$$\Lambda_A = \{(a_n)_{n\in \Z} \in \{1,2,\dots,N\}^\Z
\mid A(a_n, a_{n+1}) = 1, n\in \Z\}
$$
and
$\bar{\sigma}_A = \bar{\sigma}|_{\Lambda_A}$. 
Its one-sided shift space $X_{\Lambda_A}$ is denoted by $X_A$.
Cuntz--Krieger in \cite{CK} introduced a class of purely infinite simple $C^*$-algebras
associated with the topological Markov shifts. 
For an $N\times N$ matrix $A =[A(i,j)]_{i,j=1,\dots,N}$ with entries in $\{0,1\}$,
the Cuntz--Krieger algebra $\OA$ is defined by the universal $C^*$-algebra 
generated by $N$ partial isometries $S_1,\dots, S_N$
subject to the relations 
$\sum_{j=1}^N S_j S_j^* =1$ 
and
$S_i^* S_i = \sum_{j=1}^N A(i,j) S_j S_j^*, i=1,2,\dots, N$.
The $C^*$-subalgebra  generated by the projections of the form
$S_{i_1} \cdots S_{i_n}S_{i_n}^*\cdots S_{i_1}^*$, 
$i_1,\dots, i_n \in \{1,2,\dots,N\}$
is naturally isomorphic to the commutative $C^*$-algebra $C(X_A)$ 
of complex valued continuous functions on $X_A$, 
that is denoted by $\DA$. 
The gauge action $\rho^A$ of $\T$ on $\OA$ 
plays a crucial role analyzing the algebra 
$\OA$ that is defined by the automorphisms satisfying 
$\rho^A_t(S_j) = \exp{(2\pi\sqrt{-1}t)} S_j, j=1,2,\dots,N, t \in \R/\Z =\T$.
Among other things, Cuntz--Krieger proved in \cite{CK} 
the following result:
\begin{proposition}[{\cite[Proposition 2.17]{CK}}]\label{prop:CK}
Let $A$ and $B$ be irreducible non-permutation matrices with entries in $\{0,1\}$.
If the one-sided topological Markov shifts 
$(X_A, \sigma_A)$ and $(X_B,\sigma_B)$ 
are topologically conjugate,      
then there exists an isomorphism $\Phi:\OA\longrightarrow \OB$
of $C^*$-algebras such that 
$\Phi(\DA) = \DB$ and $\Phi\circ \rho^A_t = \rho^B_t \circ \Phi, t \in \T$.
\end{proposition}
The two-sided analogue of the implication above was also proved in 
\cite[Theorem 3.8]{CK}, and its converse was proved by
T. M. Carlsen--J. Rout \cite{CR} (cf. \cite{BC2019}) for more general setting. 
One-sided topological Markov shifts 
$(X_A, \sigma_A)$ and $(X_B,\sigma_B)$ 
are said to be eventually conjugate (\cite{MaPAMS2017}) 
if there exist a homeomorphism
$h:X_A\longrightarrow X_B$ and a nonnegative integer $K$ 
such that 
\begin{align*}
\sigma_B^K(h(\sigma_A(a)))=&  \sigma_B^{K+1}(h(a)), \qquad a \in X_A, \\
\sigma_A^K(h^{-1}(\sigma_B(b))) =&  \sigma_A^{K+1}(h^{-1}(b)), \qquad b \in X_B.\\
\end{align*}
For the converse implication of Proposition \ref{prop:CK},
the author proved the following result:
\begin{proposition}[{\cite{MaPAMS2017}}]\label{prop:MaPAMS2017}
Let $A$ and $B$ be irreducible non-permutation matrices with entries in $\{0,1\}$.
The one-sided topological Markov shifts 
$(X_A, \sigma_A)$ and $(X_B,\sigma_B)$ 
are eventually conjugate
if and only if there exists an isomorphism $\Phi:\OA\longrightarrow \OB$
of $C^*$-algebras such that 
$\Phi(\DA) = \DB$ and $ \Phi \circ \rho^A_t = \rho^B_t \circ \Phi, t \in \T$.
\end{proposition}
In \cite{BC2017}, T. M. Brix--Carlsen found an example of a pair of irreducible Markov shifts 
that are eventually conjugate but not topologically conjugate.
For an integer valued continuous map $f \in C(X_A,\Z)$ on $X_A$, 
the gauge action $\rho^{A,f}_t$ with potential $f$
is defined by the automorphisms satisfying 
$\rho^{A,f}_t(S_i) = \exp{(2 \pi \sqrt{-1} f t)}S_i, i=1,2,\dots, N$.
The action $\rho^{A,f}$ with potential $f\equiv 1$ coincides with the gauge action
$\rho^A$.
 As a complete characterization of topological conjugcy of one-sided topological Markov shifts
in terms of their Cuntz--Krieger algebras  and its gauge actions,
the author recently proved the following result:
\begin{proposition}[{\cite{MaPre2020b}}]\label{prop:MaPre2020b}
Let $A$ and $B$ be irreducible non-permutation matrices with entries in $\{0,1\}$.
The one-sided topological Markov shifts 
$(X_A, \sigma_A)$ and $(X_B,\sigma_B)$ 
are topologically  conjugate
if and only if 
there exists an isomorphism $\Phi:\OA\longrightarrow \OB$
of $C^*$-algebras such that 
$\Phi(\DA) = \DB$ and 
$$
\Phi \circ \rho^{A, f}_t= \rho^{B, \Phi(f)}_t \circ \Phi 
\qquad \text{ for all } f \in C(X_A,\Z), \, \, t \in \T.
$$
\end{proposition}
In \cite{MaPre2020a}, 
the author introduced a class of irreducible subshifts called normal subshifts.
The class of normal subshifts includes the class of 
 irreducible nontrivial shifts of finite type,
 irreducible nontrivial sofic shifts, synchronizing systems, 
Dyck shifts, Motzkin shfts, primitive substitution subshifts, $\beta$-shifts, and so on.
A lot of important  subshifts are contained in the class of normal subshifts. 
A normal subshift $\Lambda$ 
has a unique minimal presentation $\LLmin$ of $\lambda$-graph system
that is a generalization of  the Fischer cover graphs for sofic shifts.
A $\lambda$-graph system $\L = (V, E, \lambda, \iota)$ introduced in \cite{MaDocMath1999}
is a labeled Bratteli diagram $(V, E, \lambda)$ equipping with
a shift operation $\iota$, 
and it yields a $C^*$-algebra  written $\OL$ as in \cite{MaDocMath2002}.
%The class of $C^*$-algebras associated with $\lambda$-graph systems is
%a generalizatin of the class of Cuntz--Krieger algebras.
For a normal subshift $\Lambda$ and its minimal presentation $\LLmin$ of $\lambda$-graph system,
the author studied in \cite{MaPre2020a} the associated $C^*$-algebra $\mathcal{O}_{\LLmin}$ 
that is written as 
$\OLmin$.
The $C^*$-algebras are generalization of Cuntz--Krieger algebras.
The algebra $\OLmin$ contains the commutative $C^*$-algebra $\D_{\Lambda}$
isomorphic to $C(X_\Lambda)$ as a subalgebra of $\OLmin$, 
and has a natural gauge actin $\rho^\Lambda$ of $\T$.   
As a generalization of Proposition \ref{prop:MaPAMS2017},
 the following result was proved in \cite{MaPre2020a}.
\begin{proposition}[{\cite{MaPre2020a}}]\label{prop:MaPre2020a}
Let $\Lambda_1$ and $\Lambda_2$ be normal subshifts.
The one-sided  subshifts
$(X_{\Lambda_1},\sigma_{\Lambda_1})$ 
and 
$(X_{\Lambda_2},\sigma_{\Lambda_2})$ 
are eventually conjugate
if and only if 
there exists an isomorphism
$\Phi:\OLamonemin\longrightarrow\OLamtwomin$ 
of  $C^*$-algebras such that 
$\Phi({\mathcal{D}}_{\Lambda_1}) ={\mathcal{D}}_{\Lambda_2}$
and
\begin{equation*}
\Phi\circ\rho^{\Lambda_1}_t = \rho^{\Lambda_2}_t\circ\Phi
\qquad
t \in \T.
\end{equation*}
\end{proposition}

In this paper, we will prove the following theorem generalizing 
Proposition \ref{prop:MaPre2020b}.
It characterizes one-sided topological conjugacy of normal subshifts
in terms of the associated $C^*$-algebras $\OLmin$ and its gauge actions with potentials.
\begin{theorem}%[{Theorem \ref{thm:soficconj}}] 
\label{thm:main}
Let $\Lambda_1$ and $\Lambda_2$ be normal subshifts.
The following assertions are equivalent. 
\begin{enumerate}
\renewcommand{\theenumi}{\roman{enumi}}
\renewcommand{\labelenumi}{\textup{(\theenumi)}}
\item
Their one-sided  subshifts
$(X_{\Lambda_1},\sigma_{\Lambda_1})$ 
and 
$(X_{\Lambda_2},\sigma_{\Lambda_2})$ 
are topologically conjugate.
\item
There exists an isomorphism
$\Phi:\OLamonemin\longrightarrow\OLamtwomin$ 
of  $C^*$-algebras such that 
$\Phi({\mathcal{D}}_{\Lambda_1}) ={\mathcal{D}}_{\Lambda_2}$
and
\begin{equation*}
\Phi\circ\rho^{\Lambda_1, f}_t = \rho^{\Lambda_2,\Phi(f)}_t\circ\Phi
\qquad
\text{ for all } \quad f \in C(X_{\Lambda_1},\Z), \quad t \in \T,
\end{equation*}
\end{enumerate}
\end{theorem}
where
$\rho^{\Lambda_1, f}_t , \rho^{\Lambda_2,\Phi(f)}_t$
are gauge actions with potential $f$ and $\Phi(f)$, respectively.
They are defined in a similar fashion to those of Cuntz--Krieger algebras.

We have to remark here that Brix--Carlsen in \cite{BC2019} characterized topological conjugacy
of general subshifts in terms of certain $C^*$-algebras with its gauge actions and some additional structure
written $\tau$ of completely positive maps. 
The $C^*$-algebras treated in their paper \cite{BC2019} are different from our $C^*$-algebras in this paper
if the subshifts are shifts of finite type.
In case of irreducible nontrivial sofic shifts, 
our $C^*$-algebras $\OLmin$ are simple whereas their $C^*$-algebras 
in \cite{BC2019} are not simple in general.

In what follows, we denote by $\Zp$ and $\N$ 
the set of nonnegative integers and the set of positive integers, respectively.

%%%%%%%%%%%%%%%%%%%%%%%%%%%%%%%%%%%%%%%%%%%%%%%%%%%%%%%%%%%%%%%%%%%
\section{Preliminary}
%%%%%%%%%%%%%%%%%%%%%%%%%%%%%%%%%%%%%%%%%%%
%%%%%%%%%%%%%%%%%%%%%%%%%%%%%%%%%%%%%%%%%%%%

In this section, 
we will briefly review  normal sunshifts,
$\lambda$-graph systems, its \'etale groupoids  and the associated $C^*$-algebras.

\medskip

{\bf 1. Normal subshifts.}
Let $\Lambda$ be a subshift over alphabet 
$\Sigma =\{1,2,\dots, N\}$ with $N\ge 2$.
Denote by $X_\Lambda$ its right one-sided subshift.
Let us denote by $B_k(\Lambda)$ the set of admissible words 
$\{(a_1, \dots, a_k) \in \Sigma^k \mid (a_n)_{n\in \Z}\in \Lambda \}$
of $\Lambda$ with its length $k$. 
We put $B_*(\Lambda) = \cup_{k=0}^\infty B_k(\Lambda)$, where 
$B_0(\Lambda)$ denotes the empty word.
The length $m$ of  a word $(\mu_1,\dots,\mu_m) \in B_m(\Lambda)$ 
is denoted by $|\mu|$. 
For two words $\mu=(\mu_1,\dots,\mu_m),\nu=(\nu_1,\dots,\nu_n)$,
let us denote by
$\mu\nu $ its concatenation 
$(\mu_1,\dots,\mu_m,\nu_1,\dots,\nu_n)$.
For $a = (a_n)_{n \in \N} \in X_\Lambda$ and $k,l \in \N$ with $k\le l$, 
we put
$a_{[k,l]} =(a_k, a_{k+1},\dots, a_l)\in B_{l-k+1}(\Lambda),
\,a_{[k,l)} =(a_k, a_{k+1},\dots, a_{l-1})\in B_{l-k}(\Lambda),
$
and
$a_{[k,\infty)} =(a_k, a_{k+1},\dots )\in X_\Lambda.$
Let us define 
the predecessor set $\Gamma_l^-(\mu) $
and the follower set $\Gamma_l^+(\mu) $ of a word $\mu \in B_m(\Lambda)$ by 
\begin{equation*}
\Gamma_l^-(\mu) = \{ \nu \in B_l(\Lambda) \mid \nu \mu \in B_{l+m}(\Lambda)\}
\quad
\text{ and }
\quad
\Gamma_l^+(\mu) = \{ \nu \in B_l(\Lambda) \mid \mu \nu \in B_{l+m}(\Lambda)\}.
\end{equation*}
We put
$\Gamma_*^-(\mu) = \cup_{l=0}^\infty \Gamma_l^-(\mu)
$ and
$
\Gamma_*^+(\mu) = \cup_{l=0}^\infty \Gamma_l^+(\mu).
$
The notion of $l$-synchronizing word was introduced in 
\cite{KMAAA2013} (cf.  \cite{MaIsrael2013}, \cite{MaJAMS2013})
in the following way.
 A word $\mu \in B_*(\Lambda)$ is said to be $l$-{\it synchronizing}\/
if the equality
$\Gamma_l^-(\mu) = \Gamma_l^-(\mu \xi)$ holds
for any word $\xi \in \Gamma_*^+(\mu)$. 
We denote by $S_l(\Lambda)$ 
the set of $l$-synchronizing words of $\Lambda$.
Any admissible word of $\Lambda$ 
is defined to be $0$-synchronizing so that $S_0(\Lambda) = B_*(\Lambda)$. 
Recall that a subshift $\Lambda$ is said to be irreducible if for any ordered pair 
$\mu, \nu $ of $B_*(\Lambda)$, there exists a word $\eta\in B_*(\Lambda)$
such that $\mu \eta \nu \in B_*(\Lambda)$ (cf. \cite{LM}).
An  irreducible subshift $\Lambda$ is said to be  
$\lambda$-{\it synchronizing}\/ 
if  for any word $\nu \in B_l(\Lambda)$
and an integer $k \in \N$ with $k \ge l$, 
there exists $\mu \in S_{k}(\Lambda)$ such that  
$\nu \mu \in S_{k-l}(\Lambda)$ (\cite{KMAAA2013},  \cite{MaIsrael2013}, \cite{MaJAMS2013}).
A subshift $\Lambda$ is said to be {\it nontrivial}\/ if the cardinality of $\Lambda$ is not finite as a set.

\begin{definition}[{\cite{MaPre2020a}}]
A nontrivial subshift $\Lambda$ is said to be  {\it normal}\/
if it is irreducible and $\lambda$-synchronizing. 
\end{definition}
The class of normal subshifts 
contains the class of irreducible nontrivial sofic shifts $\Lambda$,
and hence it contains topological Markov shifts for irreducible non-permutation matrices.
It also includes a lot of nonsofic subshifts such as 
Dyck shifts, substitution minimal shifts, $\beta$-shitfs and so on.
See \cite{MaPre2020a} for the detail of normal subshifts. 

\medskip

%%%%%%%%%%%%%%%%%%%%%%%%%%%%%%%%%%%%%%%%%%%%%%%%%%%%%
{\bf 2. $\lambda$-graph system $\L$.}
We will next recall the notion of $\lambda$-graph system
introduced in \cite{MaDocMath1999}. 
A $\lambda$-graph system $\L =(V, E, \lambda,\iota)$ over $\Sigma$
consists of a sequence of vertex set $V = \cup_{l=0}^\infty V_l$,
 edge set $E = \cup_{l=0}^\infty E_{l,l+1}$, 
 a labeling map $\lambda: E \longrightarrow \Sigma$ 
and a surjective map $\iota: V_{l+1} \longrightarrow V_l, l \in \Zp$
that satisfy some compatibility condition called the local property of 
$\lambda$-graph system 
(see \cite{MaIsrael2013}, \cite{MaJAMS2013}, \cite{MaPre2020a}
for the local property of $\lambda$-graph system).
Each edge $e \in E_{l,l+1}$ has its source $s(e) $ in $V_l$,
its terminal $t(e)$ in $V_{l+1}$ and its label $\lambda(e)$ in $\Sigma$.
The first three objects $(V, E, \lambda)$
yield a labeled Bratteli diagram.
The fouth one $\iota$ plays a role of shift on the labeled Bratteli diagram.
Any $\lambda$-graph system $\L$ presents a subshift written $\Lambda_{\L}$
whose admissible words $B_*(\Lambda_\L)$ 
 consist of finite label sequences appearing in 
the labeled Bratteli diagram of the $\lambda$-graph system.   
%The subshift $\Lambda_\L$ is now denoted by $\Lambda$ for brevity. 
A $\lambda$-graph system $\L$ is said to be {\it left-resolving}\/ if 
two edges $e, f \in E_{l,l+1}$
satisfy
$\lambda(e) = \lambda(f)$ and $t(e) = t(f),$
then $e=f$.
For a vertex $v \in V_l$, let us define the predecessor set $\Gamma_l^-(v)$  of $v$ by
\begin{align*}
 & \Gamma_l^-(v) \\ 
=& \{ (\lambda(e_1),\dots,\lambda(e_l)) \in B_l(\Lambda_\L)
\mid e_i \in E_{i-1,i}, t(e_i) = s(e_{l+1}), i=1,2,\dots,l-1, \, t(e_l)= v\},
\end{align*}
that is 
the set of words of $B_l(\Lambda_\L)$
of label sequences 
ending at the vertex $v$.
The $\lambda$-graph system $\L$ is said to be {\it predecessor-separated}\/ if
every distinct vertices  $v, u \in V_l$ must satisfy
$\Gamma_l^-(v) \ne \Gamma_l^-(u). $
Any subshift $\Lambda$  may be presented by several kinds of different $\lambda$-graph systems. 
In particular if a subshift $\Lambda$ is normal, 
it has a minimal presentation $\L_\Lambda^\min$
of $\lambda$-graph system.
Let us recall its construction following \cite{MaIsrael2013}.
Two words $\mu, \nu \in S_l(\Lambda)$ are said to be $l$-past equivalent if
$\Gamma_l^-(\mu) = \Gamma_l^-(\nu)$.
The equivalence class is denoted by $[\mu]_l$.
Let  
$V_l^\min$ be the set of $l$-past equivalence classes
of $S_l(\Lambda)$. 
That will be a vertex set of the $\lambda$-graph system $\LLmin$.
For $\nu \in S_{l+1}(\Lambda)$ and $\alpha \in \Sigma$,
we define an edge from the vertex $[\alpha\nu]_l \in V_l^\min$ 
to the vertex $[\nu]_{l+1} \in V_{l+1}^\min$
that is labeled  $\alpha$.
The set of such edges is denoted by $E^\min_{l,l+1}$.
A word $\nu \in S_{l+1}(\Lambda)$ naturally defines a correspondence 
from $[\nu]_{l+1}\in V_{l+1}^\min$ to $[\nu]_l \in V_l^\min$,
that yields a surjective map
$\iota^\min:V_{l+1}^\min \longrightarrow V_l^\min$.
We have a $\lambda$-graph system 
$\L_\Lambda^\min =(V^\min, E^\min, \lambda^\min, \iota^\min)$ over $\Sigma$
that presents the original subshift $\Lambda$.
The $\lambda$-graph system 
$\L_\Lambda^\min$ is left-resolving and predecessor-separated.
Since $\L_\Lambda^\min$ does not
have  any proper  $\lambda$-graph subsystem presenting the original subshift $\Lambda$
(\cite[Proposition 3.10]{MaIsrael2013}), 
it is called the minimal presentation of $\Lambda$.

\medskip

{\bf 3. \'Etale groupoid $G_\L$ from $\L$.}
To construct a $C^*$-algebra from a left-resolving $\lambda$-graph system $\L$, 
we provide an \'etale groupoid written $G_\L$ from $\L$.  
Suppose that a $\lambda$-graph system $\L$ is left-resolving.
Let $\Omega_\L$ be the set of projective  limit of the system
$\iota: V_{l+1} \longrightarrow V_l, l \in \Zp$, that is defined by 
\begin{equation*}
\Omega_\L = \{ (v^l)_{l\in \Zp} \in 
\prod_{l\in \Zp} V_l \mid \iota(v^{l+1}) = v^l, l \in \Zp\}.
\end{equation*}
It is endowed with its projective limit topology, so that 
it is a compact Hausdorff space.
Let $E_\L$ be the set of triples 
$(u, \alpha, v) \in \Omega_\L\times\Sigma\times\Omega_\L$
such that  
there exists a sequence of edges 
$e_{l,l+1} \in E_{l,l+1}$ such that 
$$
s(e_{l,l+1}) = u_l, \qquad
t(e_{l,l+1}) = v_{l+1}, \qquad  
\lambda(e_{l,l+1}) = \alpha,
\qquad l\in \Zp,
$$
where $u=(u^l)_{l\in \Zp}, v=(v^l)_{l\in \Zp}\in \Omega_\L$.
It becomes a continuous graph in the sense of V. Deaconu
(cf. \cite{De}, \cite{De2}).
Let $X_\L$ be the set of one-sided paths of $E_\L$ defined by 
\begin{align*}
X_\L = \{ (\alpha_i, v_i)_{i\in \N} \in \prod_{i\in \N}(\Sigma\times\Omega_\L) & \mid
(v_i, \alpha_{i+1}, v_{i+1})\in E_\L \text{ for all } i\in \N, \\
& (v_0, \alpha_1, v_1) \in E_\L \text{ for some } v_0\in \Omega_\L\}. 
\end{align*}
It is endowed with its relative topology from
 $\prod_{i\in \N}(\Sigma\times\Omega_\L)$
 so that it becomes a zero-dimensional compact Hausdorff space.
The shift map
$\sigma_\L : X_\L \longrightarrow X_\L$ is defined by 
$\sigma_\L((\alpha_i, u_i)_{i\in \N}) =
(\alpha_{i+1}, u_{i+1})_{i\in \N}.
$
It gives rise to a continuous surjection on $X_\L$.
Let us denote by $\Lambda$ the presented subshift
$\Lambda_\L$ by $\L$.
Its one-sided subshift is denoted by $X_\Lambda$. 
The map $\pi_\L: X_\L\longrightarrow X_\Lambda$ 
defined by 
\begin{equation}
\pi_\L((\alpha_i, v_i)_{i\in \N}) = (\alpha_i)_{i\in \N}, \qquad (\alpha_i, v_i)_{i\in \N}\in X_\L
\label{eq:pilambda}
\end{equation}
is a continuous surjection satisfying 
$\pi_\L\circ\sigma_\L = \sigma_\Lambda\circ\pi_\L$.

Let us denote by
$G_\L$ so-called the Deaconu-Renault groupoid  for the shift dynamical system
$(X_\L, \sigma_\L)$
(cf. \cite{De}, \cite{De2}, \cite{Renault}, \cite{Renault2000}, \cite{Renault2}, \cite{Renault3}, \cite{Sims}).
It is defined by 
\begin{equation*}
G_\L = \{ (x, p-q, z) \in X_\L\times\Z\times X_\L \mid 
\sigma_\L^p(x) = \sigma_\L^q(z) \}.
\end{equation*}
The open neighborhood basis  of $G_\L$
consists of the set of the form
$
\{ (x, p-q, z) \in G_\L \mid 
x \in U, z \in V, \sigma_\L^p(x) = \sigma_\L^q(z) \}
$
 for open sets $U, V $ of $X_\L$ such that 
 $\sigma_\L^p|_U$ and $\sigma_\L^q|_V$ are homeomorphisms 
 on the same open range.
Its unit space 
$G_\L^{(0)} =\{(x,0,x) \in G_\L \mid x \in X_\L\}$ 
is naturally identified 
with the original space $X_\L$ as topological spaces.
The source map and  range map of $G_\L$ are defined by
$s(x,n, z) =z, r(x,n,z) =x$. 
The product and inverse operations are defined by
$(x,n,z)(z,m,w) = (x, n+m, w)$
and
$(x,n,z)^{-1} = (z,-n,x)$.
It is known that the groupoid $G_\L$ is \'etale and amenable 
(cf. \cite{MaDocMath2002}, \cite{MaDynam2020}, \cite{Renault2}, \cite{Renault3}, \cite{Sims}).

\medskip

%%%%%%%%%%%%%%%%%%%%%%%%%%%%%%%%%%%%%%%%%
{\bf 4. $C^*$-algebra $\OL$.}
  Let us briefly review the construction of the $C^*$-algebra
 $C^*(G_\L)$ of the \'etale groupoid $G_\L$.  
Let $C_c(G_\L)$ be the $*$-algebra of complex valued continuous functions on $G_\L$ with compact support.
Its product structure and $*$-involution are defined by  
\begin{gather*}
(f*g)(t) = \sum_{s\in G_\L, r(s) = r(t)} f(s)g(s^{-1}t),\qquad
f^*(t) = \overline{f(t^{-1})}
\end{gather*} 
for $f, g \in C_c(G_\L), t \in G_\L$.
For $\xi, \eta \in C_c(G_\L)$ and $f\in C(G_\L^{(0)})$, 
by setting
\begin{gather*}
(\xi f)(x,n,z) = \xi(x,n,z) f(z), \qquad
<\xi, \eta>(z) = \sum_{(x,n,z) \in G_\L} \overline{\xi(x,n,z)}\eta(x,n,z)
\end{gather*}
the $*$-algebra $C_c(G_\L)$ has $C(G_\L^{(0)})$-right module with 
$C(G_\L^{(0)})$-valued inner product.
The completion denoted by $\ell^2(G_\L)$
of $C_c(G_\L)$ by the inner product becomes a Hilbert $C^*$-right module over 
the commutative $C^*$-algebra $C(G_\L^{(0)})$. 
For $f \in C_c(G_\L)$, 
a bounded adjointable $C(G_\L^{(0)})$-module map
$\pi(f)$ on $\ell^2(G_\L)$ is defined by setting
$\pi(f)\xi = f* \xi$ 
for $ \xi \in C_c(G_\L)$.
Then the reduced $C^*$-algebra $C^*_r(G_\L)$ of the groupoid $G_\L$ is defined by 
the closure of $\pi(C_c(G_\L))$ in 
the $C^*$-algebra  $B(\ell^2(G_\L))$ of bounded adjointable $C_0(G_\L^{(0)})$-module maps
on $\ell^2(G_\L).$
\begin{definition}[{\cite{MaDocMath2002}}]
The $C^*$-algebra $\OL$ is defined by the reduced $C^*$-algebra $C^*_r(G_\L)$
of the groupoid $G_\L$.
\end{definition}
Let us denote by $\{ v_1^l, \dots, v_{m(l)}^l\}$ 
the set $V_l$ of vertices at level $l$ in the $\lambda$-graph system
$\L$.  
For $x =(\alpha_i, v_i)_{i\in \N}\in X_\L,$
we set 
$\lambda_i(x) =\alpha_i \in \Sigma, i\in \N$.
There exists $v(x)_0 ={(v(x)_0^l)}_{l\in \Zp}\in \Omega_\L$ such that 
$(v(x)_0, \alpha_1, v_1) \in E_\L$. 
For $\alpha \in \Sigma, v_i^l \in V_l$, let us define clopen sets 
$U(\alpha), U(v_i^l) \subset G_\L$ by setting
\begin{gather*}
U(\alpha) = \{ (x, 1, \sigma_\L(x)) \in G_\L \mid \lambda_1(x) = \alpha\}, \\ 
U(v_i^l) = \{ (x,0,x) \in G_\L \mid v(x)_0^l = v_i^l \}.
\end{gather*}
Let us denote by $\chi_F\in C_c(G_\L)$ 
the characteristic function of a clopen set $F\subset G_\L$.
We define elements $S_\alpha, E_i^l$ of $\OL$ by setting 
\begin{equation*}
S_\alpha = \pi(\chi_{U(\alpha)}), \qquad 
E_i^l = \pi(\chi_{U(v_i^l)}) \qquad \text{ in } \pi(C_c(G_\L)). 
\end{equation*}
For the $\lambda$-graph system $\L$, its transition matrix system
$(A_{l,l+1}, I_{l,l+1})_{l\in \Zp}$ 
defined below determines the structure 
of $\L$: 
\begin{align*}
A_{l,l+1}(i, \alpha,j) & =
{\begin{cases}
1 & \text{ if there exists } e \in E_{l,l+1} ; s(e) = v_i^l, t(e)= v_j^{l+1}, \lambda(e) = \alpha, \\
0 & \text{ otherwise,}
\end{cases}} \\
I_{l,l+1}(i,j) & =
{\begin{cases}
1 & \text{ if  } \iota(v_j^{l+1}) = v_i^l,  \\
0 & \text{ otherwise}
\end{cases}} \\
\end{align*}
for $v_i^l \in V_l, v_j^{l+1} \in V_{l+1}, \alpha \in \Sigma.$
 A $\lambda$-graph system $\L$ is said to satisfy {\it condition}\/ (I)
if for every vertex $v \in V$, the set $\Gamma_\infty^+(v)$ 
of infinite concatenated label sequences leaving $v$ in $\L$ 
contains at least two distinct sequences (\cite{MaDocMath2002}).  
For $\mu =(\mu_1,\dots,\mu_m) \in B_m(\Lambda)$, we write
$S_\mu = S_{\mu_1}\cdots S_{\mu_m}$.

\begin{proposition}[{\cite[Theorem 3.6, Theorem 4.3]{MaDocMath2002}}]
Let $\L$ be a left-resolving $\lambda$-graph system.
The $C^*$-algebra $\OL$ is realized as a universal $C^*$-algebra generated by partial isometries
$S_\alpha, \alpha \in \Sigma$ and projections
$E_i^l, i=1,2,\dots,m(l)$
subject to the relations
\begin{gather*}
\sum_{\beta \in \Sigma} S_\beta S_\beta^* = \sum_{j=1}^{m(l)} E_j^l = 1, \qquad
S_\alpha S_\alpha^* E_i^l = E_i^l S_\alpha S_\alpha^*, \\
E_i^l = \sum_{j=1}^{m(l+1)} I_{l,l+1}(i,j) E_j^{l+1}, \qquad 
S_\alpha^* E_i^l S_\alpha = \sum_{j=1}^{m(l+1)} A_{l,l+1}(i,\alpha, j) E_j^{l+1}, 
\end{gather*}
for $\alpha \in \Sigma, i=1,2,\dots,m(l).$
If in particular $\L$ satisfies  condition (I),
the $C^*$-algebra $\OL$ is a unique $C^*$-algebra subject to the above operator relations.
Furthermore if $\L$ is $\lambda$-irreducible in the sense of \cite{MaJAMS2013},
it is simple and purely infinite.
\end{proposition}  
We also note that if $\L$ is predecessor-separated, 
each projection $E_i^l$ of generators 
may be written by some products of $S_\mu^* S_\mu$ and $1- S^*_\nu S_\nu$ for
$\mu, \nu \in B_l(\Lambda)$
in the following way:
%%%%%%%%%%%%%%%%%%%%%%%%%%%%
\begin{equation}
E_i^l = \prod_{\mu \in B_l(\Lambda)} S_\mu^* S_\mu^{\epsilon_i^l(\mu)} \label{eq:Eil}
\end{equation}
where
\begin{equation*}
S_\mu^* S_\mu^{\epsilon_i^l(\mu)}
=
\begin{cases}
S_\mu^* S_\mu & \text{ if } \mu \in \Gamma_l^-(v_i^l),\\
1-S_\mu^* S_\mu & \text{ otherwise, }
\end{cases}
\end{equation*}
%%%%%%%%%%%%%%%%%%%%%
so that the algebra $\OL$
is  generated by only the partial isometries
$S_1,\dots,S_N$.  

Let us define commutative $C^*$-subalgebras
$\DL, \DLam$ of $\OL$ by
\begin{align}
\DL & = C^*(S_\mu E_i^l S_\mu^* \mid \mu \in B_*(\Lambda), i=1,2,\dots, m(l), l\in \Zp),
\label{eq:DL} \\
\DLam & = C^*(S_\mu S_\mu^* \mid \mu \in B_*(\Lambda)), 
\label{eq:DLam} 
\end{align}
where $C^*(F)$ for a subset $F \subset \OL$ 
denotes the $C^*$-subalgebra of $\OL$ generated by 
$F$. 
It is easy to see that 
the $C^*$-subalgebras $\DL, \DLam$
are naturally isomorphic to the commutative $C^*$-algebras
$C(X_\L), C(X_{\Lambda})$ of continuous functions on $X_\L, X_{\Lambda}$,
respectively.
We note that the natural inclusion relations  $\DLam\subset \DL $
comes from the continuous surjection 
$\pi_\L: X_\L\longrightarrow X_\Lambda$.
\begin{proposition}[{\cite[Proposition 3.7]{MaPre2020a}}]
Let $\L$ be a left-resolving $\lambda$-graph system satisfying condition (I).
Then  
$\DLam^\prime \cap \OL = \DL,$
where $\DLam^\prime \cap \OL$ denotes the algebra of elements of $\OL$
commuting with all elements of $\DLam.$ 
\end{proposition}

\medskip

%%%%%%%%%%%%%%%%%%%%%%%%%%%%%%%%%%%%%%%%%%%%%%%%%%%%%%%%%%%%%%%%%%%%%
{\bf 5. The $C^*$-algebra $\OLmin$ of a normal subshift $\Lambda$.}
%%%%%%%%%%%%%%%%%%%%%%%%%%%%%%%%%%%%%%%%%%%%%%%%%%%%%%%%%%%%%%%%%%%%%%%%%
Let $\Lambda$ be a normal subshift.
Let $\LLmin$ be its minimal presentation of $\Lambda$
with transition matrix system  
$(A_{l,l+1}^\min, I_{l,l+1}^\min)_{l\in \Zp}$.
We note that the commutative $C^*$-subalgebra $\DLmin$
of $\OLmin$ is defined in \eqref{eq:DL} for $\L = {\frak L}^\min_\Lambda$.
Since the $\lambda$-graph system $\LLmin$ satisfies condition (I) (\cite{MaPre2020a}),  
we have the following proposition.
\begin{proposition}[{\cite[Proposition 3.9]{MaPre2020a}}] \label{prop:2.4}
The $C^*$-algebra $\OLmin$ is realized as a universal unique
$C^*$-algebra generated by partial isometries
$S_\alpha, \alpha \in \Sigma$ and projections
$E_i^l, i=1,2,\dots,m(l)$
subject to the relations
\begin{gather*}
\sum_{\beta \in \Sigma} S_\beta S_\beta^* = \sum_{j=1}^{m(l)} E_j^l = 1, \qquad
S_\alpha S_\alpha^* E_i^l = E_i^l S_\alpha S_\alpha^*, \\
E_i^l = \sum_{j=1}^{m(l+1)} I_{l,l+1}^\min(i,j) E_j^{l+1}, \qquad 
S_\alpha^* E_i^l S_\alpha = \sum_{j=1}^{m(l+1)} A_{l,l+1}^\min(i,\alpha, j) E_j^{l+1}, 
\end{gather*}
for $\alpha \in \Sigma, i=1,2,\dots,m(l).$
It satisfies $\DLam^\prime\cap\OLmin = \DLmin,$
where $\DLam^\prime\cap\OLmin$ is the $C^*$-subalgebra of $\OLmin$
consisting of elements of $\OLmin$ commuting with all elements of $\DLam$.
\end{proposition}
The $C^*$-algebra $\OLmin$ is generated by only the partial isometries 
$S_1,\dots, S_N$, because $\LLmin$ is predecessor-separated,
so that the projection $E_i^l$ is written by using $S_1, \dots, S_N$ 
as in \eqref{eq:Eil}.

%%%%%%%%%%%%%%%%%%%%%%%%%%%%%%%%%%%%%%%%%%%%%%%%%%%%%%%%%%%%%%%%%%%%%%%%%%%%%
%%%%%%%%%%%%%%%%%%%%%%%%%%%%%%%%%%%%%%%%%%%%%%%%%%%%%%%%%%%%%%%%%%%%%%%%%%%%%%%
\section{Proof of Theorem \ref{thm:main} (i) $\Longrightarrow$ (ii)} 
%%%%%%%%%%%%%%%%%%%%%%%%%%%%%%%%%%%%%%%%%%%%%%%%%%%%%%%%%%%%%%%%%%%%%%%%
%%%%%%%%%%%%%%%%%%%%%%%%%%%%%%%%%%%%%%%%%%%%%%%%%%%%%%%%%%%%%%%%%%%%%
In this setion, we will give a proof of 
Theorem \ref{thm:main} (i) $\Longrightarrow$ (ii).
Let us fix a left-resolving $\lambda$-graph system $\L$.
Consider the shift dynamical system $(X_\L,\sigma_\L)$
defined in the previous section.
For $f \in C(X_\L,\Z)$ and $n \in \Zp$, define 
$f^n \in C(X_\L,\Z)$ by setting for $x \in X_\L$
\begin{equation*}
f^n(x) := 
\begin{cases}
\sum_{i=0}^{n-1}f(\sigma_\L^i(x)) & \text{ for } n\ge 1, \\
0 & \text{ for } n =0.
\end{cases}
\end{equation*}
It is straightforward to check that the identity
\begin{equation*}
f^{n+m}(x) = f^n(x) + f^m(\sigma_\L^n(x)), \qquad x \in X_\L, \quad n,m \in \Zp
\end{equation*}
holds.
We provide several lemmas.
\begin{lemma}\label{lem:3.1}
For $f \in C(X_{\frak L},\Z)$, define a function 
$f_{G_{\frak L}}:G_{\frak L}\longrightarrow \Z$ by setting
\begin{equation}\label{eq:fGL}
f_{G_{\frak L}}(x, p-q,z) := f^p(x) - f^q(z),  \qquad (x, p-q,z)  \in G_{\frak L}.
\end{equation} 
Then $f_{G_{\frak L}}:G_{\frak L}\longrightarrow \Z$
is a continuous groupoid homomorphism.
\end{lemma}
\begin{proof}
We will first see that the function $f_{G_{\frak L}}$ is well-defined.
Suppose that 
$ (x, p-q,z)=(x, p' -q', z)  \in G_{\frak L}$
so that 
$\sigma_{\L}^p(x) =\sigma_{\L}^q(z), \sigma_{\L}^{p'}(x) =\sigma_{\L}^{q'}(z).$
One may assume that 
$p\ge p'$ and $q\ge q'$ so that we have
$p-p' = q - q'.$
We then have
\begin{align*}
 &(f^p(x) - f^q(z)) -(f^{p'}(x) - f^{q'}(z)) \\
%=& (f^p(x) - f^{p'}(x)) -(f^q(z) - f^{q'}(z))  =
%\sum_{i=p'}^{p-1}f(\sigma_\L^i(x)  - \sum_{j=q'}^{q-1}f(\sigma_\L^j(z))  \\
=&
 \sum_{i=p'}^{p-1}f(\sigma_\L^i(x))  - \sum_{j=q'}^{q-1}f(\sigma_\L^j(z))  
 =
\sum_{i=0}^{p-p'-1}f(\sigma_\L^{p'+i}(x))  - \sum_{j=0}^{q-q'-1}f(\sigma_\L^{q'+j}(z)) =0.
\end{align*}
We will next show that the function 
$f_{G_{\frak L}}:G_{\frak L}\longrightarrow \Z$
is a groupoid homomorphism.
For 
$ (x, p-q,z), (z, r-s, w)  \in G_{\frak L}$, we have 
$ (x, p-q,z)\cdot (z, r-s, w) =(x,p-q+r -s,w)$ and
\begin{align*}
& f_{G_\L}(x, p-q,z) + f_{G_\L} (z, r-s, w) \\
=& (f^p(x) - f^q(z) ) + ( f^r(z) - f^s(w) )\\
=& \{ (f^{p+r}(x)- f^r(\sigma_\L^p(x)))
      - (f^{q+r}(z)-f^r(\sigma_\L^q(z)))\} \\
&+
\{(f^{r+q}(z) -f^q(\sigma_\L^r(z))) - (f^{q+s}(w)- f^q(\sigma_\L^s(w)))\}\\
=& f^{p+r}(x) - f^{q+s}(w) =f_{G_\L}(x, (p+r)-(q+s), w).
\end{align*}
As its continuity is clear, the map
 $f_{G_{\frak L}}:G_{\frak L}\longrightarrow \Z$
gives rise to a continuous groupoid homomorphism.
\end{proof}

%As in \cite{MaDynam2020}, the following lemma holds.
%\begin{lemma}[{ cf. \cite[Lemma 3.2]{MaDynam2020}}]
%In the $C^*$-algebra $\OL$, we have 
%\begin{equation}
%\rho^{\L,f}_t(S_\mu E_i^l S_\nu^*)=
%\exp{(2\pi\sqrt{-1}(f^{|\mu|} - f^{|\nu|})t)}\cdot S_\mu E_i^l S_\nu^*,
% \qquad \mu, \nu \in B_*(\Lambda).
%\end{equation}
%\end{lemma}
%\begin{proof} By definition, it is routine to check that the equalities
%$\rho^{\L,f}_t(S_\mu)= U_t(f^{|\mu|}) S_\mu$ holds (cf. \cite[Lemma 3.1]{MaMZ2016}).
%Since $\rho^{\L,f}_t(E_i^l)= E_i^l,$ we obtain the desired identity.
%\end{proof}
%%%%%%%%%%%%%%%%%%%%%%%%%%%%%%
The following lemma is a direct consequence of Lemma \ref{lem:3.1}.
\begin{lemma}[{\cite[Proof of (i) $\Longrightarrow$ (ii) of Proposition 4.2]{MaDynam2020}}]
Let ${\frak L}_1,{\frak L}_2$ be left-resolving $\lambda$-graph systems.
Assume that there exist a homeomorphism
$h_{\L}: X_{{\frak L}_1}\longrightarrow X_{{\frak L}_2}$
and continuous maps
$k_i, l_i: X_{{\frak L}_i}\longrightarrow \Z, i=1,2$
satisfying
\begin{align}
\sigma_{\L_2}^{k_1(x)}(h_\L(\sigma_{\L_1}(x)))
 =\sigma_{\L_2}^{l_1(x)}(h_\L(x)), \qquad x \in X_{\L_1}, \label{eq:sigmaL2x}\\
\sigma_{\L_1}^{k_2(y)}(h_\L^{-1}(\sigma_{\L_2}(y)))
 =\sigma_{\L_1}^{l_2(y)}(h_\L^{-1}(y)), \qquad y \in X_{\L_2}. \label{eq:sigmaL1y}
\end{align}
Put $c_1(x) = l_1(x) - k_1(x)$
and $c_1^p(x) = \sum_{n=0}^{p-1}c_1(\sigma_{\L_1}^n(x))$ for $x \in X_{\L_1}.$ 
Then the map 
\begin{equation}\label{eq:varphi}
\varphi: (x,p-q,z) \in G_{\L_1}\longrightarrow (h_{\L}(x), c_1^p(x) - c_1^q(z), h_{\L}(z)) \in G_{\L_2}
\end{equation} 
gives rise to an isomorphism of \'etale groupoids such that 
the restriction $\varphi|_{G_{\L_1}^{(0)}} : G_{\L_1}^{(0)}\longrightarrow G_{\L_2}^{(0)}$
of $\varphi: G_{\L_1}\longrightarrow G_{\L_2}$ to  the unit space
$G_{\L_1}^{(0)}$ coincides with $h_\L: X_{\L_1}\longrightarrow X_{\L_2}$
under the natural identification between $G_{\L_i}^{(0)}$ and $X_{\L_i}, i=1,2$.
\end{lemma}
Keep the above situation.
Similarly to \cite[Lemma 5.1]{MaPacific2010},
the identities for $p\in \N$ below 
follow from \eqref{eq:sigmaL2x}, \eqref{eq:sigmaL1y} by induction on $p$:
\begin{align}
\sigma_{\L_2}^{k_1^p(x)}(h_\L(\sigma_{\L_1}^p(x)))
 =\sigma_{\L_2}^{l_1^p(x)}(h_\L(x)), \qquad x \in X_{\L_1}, \label{eq:sigmaL2px}\\
\sigma_{\L_1}^{k_2^p(y)}(h_\L^{-1}(\sigma_{\L_2}^p(y)))
 =\sigma_{\L_1}^{l_2^p(y)}(h_\L^{-1}(y)), \qquad y \in X_{\L_2}.\label{eq:sigmaL1py}
\end{align}
Define a map
%\begin{equation*}
$
\Psi_{h_\L}:C(X_{\L_2},\Z) \longrightarrow C(X_{\L_1},\Z)
$
%\end{equation*}
by
\begin{equation*}
\Psi_{h_\L}(f)(x) = 
\sum_{i=0}^{l_1(x)} f(\sigma_{\L_2}^i(h_\L(x)))
-
\sum_{j=0}^{k_1(x)} f(\sigma_{\L_2}^j(h_\L(\sigma_{\L_1}(x)))),
\qquad f \in C(X_{\L_2},\Z), \, \,  x \in X_{\L_1}.
\end{equation*}

\begin{lemma}
Keep the above notation.
For $p \in \N$, we have
\begin{equation}
\Psi_{h_\L}(f)^p(x) = f^{l_1^p(x)}(h_\L(x))
- f^{k_1^p(x)} (h_\L(\sigma_{\L_1}^p(x))),
\qquad f \in C(X_{\L_2},\Z), \quad x \in X_{\L_1}. \label{eq:Psihlambdafp}
\end{equation}
\end{lemma}
\begin{proof}
We fix $f \in C(X_{\L_2},\Z)$ and $x \in X_{\L_1}.$
We will prove the equality \eqref{eq:Psihlambdafp} by induction on $p$.
For $p=1$, the desired equality follows by definition
together with \eqref{eq:sigmaL2x}.
Suppose that the equality holds for $p$. 
We then have
\begin{align*}
\Psi_{h_\L}(f)^{p+1}(x)
=& \Psi_{h_\L}(f)^p(x) + \Psi_{h_\L}(f)(\sigma_{\L_1}^p(x)) \\
=& \{ f^{l_1^p(x)}(h_\L(x))
- f^{k_1^p(x)} (h_\L(\sigma_{\L_1}^p(x))) \} \\
 & + \{ f^{l_1(\sigma_{\L_1}^p(x))}(h_\L(\sigma_{\L_1}^p(x)))
- f^{k_1(\sigma_{\L_1}^p(x))} (h_\L(\sigma_{\L_1}^{p+1}(x)))\}. 
\end{align*}
On the other hand, we have 
\begin{align*}
 & f^{l_1^{p+1}(x)}(h_\L(x))
- f^{k_1^{p+1}(x)} (h_\L(\sigma_{\L_1}^{p+1}(x))) \\
= &
f^{l_1^p(x) + l_1(\sigma_{\L_1}^p(x))}(h_\L(x))
- f^{k_1^p(x) + k_1(\sigma_{\L_1}^p(x))} (h_\L(\sigma_{\L_1}^{p+1}(x))) \\
= &
f^{l_1^p(x)}(h_\L(x)) +
f^{l_1(\sigma_{\L_1}^p(x))}(\sigma_{\L_2}^{l_1^p(x)}(h_\L(x))) \\
& - \{ 
f^{k_1(\sigma_{\L_1}^p(x))} (h_\L(\sigma_{\L_1}^{p+1}(x))) 
+ 
f^{k_1^p(x)}(\sigma_{\L_2}^{k_1(\sigma_{\L_1}^p(x))}(h_\L(\sigma_{\L_1}^{p+1}(x)))) \}. 
\end{align*}
By using \eqref{eq:sigmaL2px} and \eqref{eq:sigmaL1py}, 
the second term minus the fourth term above goes to 
\begin{align*}
& f^{l_1(\sigma_{\L_1}^p(x))}(\sigma_{\L_2}^{l_1^p(x)}(h_\L(x)))
- f^{k_1^p(x)}(\sigma_{\L_2}^{k_1(\sigma_{\L_1}^p(x))}(h_\L(\sigma_{\L_1}^{p+1}(x)))) \\
= &  f^{l_1(\sigma_{\L_1}^p(x))}(\sigma_{\L_2}^{k_1^p(x)}(h_\L(\sigma_{\L_1}^p(x))))
- f^{k_1^p(x)}(\sigma_{\L_2}^{l_1(\sigma_{\L_1}^p(x))}(h_\L(\sigma_{\L_1}^{p}(x)))) \\
= & \{ f^{l_1(\sigma_{\L_1}^p(x)) + k_1^p(x)}(h_\L(\sigma_{\L_1}^p(x)))
      - f^{k_1^p(x))}(h_\L(\sigma_{\L_1}^p(x)))\} \\
 & -\{ f^{k_1^p(x)+ l_1(\sigma_{\L_1}^p(x))}(h_\L(\sigma_{\L_1}^{p}(x))) 
       -f^{l_1(\sigma_{\L_1}^p(x))}(h_\L(\sigma_{\L_1}^{p}(x))) \} \\
= & f^{l_1(\sigma_{\L_1}^p(x))}(h_\L(\sigma_{\L_1}^{p}(x))) - f^{k_1^p(x))}(h_\L(\sigma_{\L_1}^p(x))).      
\end{align*}
 Therefore we have 
\begin{equation*}
\Psi_{h_\L}(f)^{p+1}(x) = f^{l_1^{p+1}(x)}(h_\L(x))
- f^{k_1^{p+1}(x)} (h_\L(\sigma_{\L_1}^{p+1}(x))).
\end{equation*}
\end{proof}
\begin{lemma}\label{lem:3.4}
Keep the above notation.
Let $\varphi: G_{\L_1}\longrightarrow G_{\L_2}$
be the groupoid isomorphism defined by \eqref{eq:varphi}.
For $f \in C(X_{\L_2},\Z)$, 
let $f_{G_\L}:G_{\L_2}\longrightarrow \Z$
be the groupoid homomorphism defined by \eqref{eq:fGL}.
Then we have the equality
\begin{equation}
(f_{G_\L}\circ \varphi)(x, p-q,z) = \Psi_{h_\L}(f)^p(x) - \Psi_{h_\L}(f)^q(z), 
\qquad (x,p-q,z) \in G_{\L_1}.
\end{equation}
Hence we have 
$f_{G_\L}\circ \varphi = \Psi_{h_\L}(f)_{G\L}.$
\end{lemma}
\begin{proof}
For $(x,p-q,z) \in G_{\L_1}$, we have
\begin{align*}
 & (f_{G_\L}\circ \varphi)(x, p-q,z) \\
=&
f_{G_\L}(h_{\L}(x), c_1^p(x)-c_1^q(z), h_{\L}(z)) \\ 
=&
f_{G_\L}(h_{\L}(x), (l_1^p(x)+k_1^q(z) )- ( l_1^q(z) +k_1^p(x) ), h_{\L}(z)) \\ 
=&
f^{ l_1^p(x)+k_1^q(z) }(h_{\L}(x)) - f^{ l_1^q(z) +k_1^p(x) }(h_{\L}(z)) \\ 
=&
\{f^{l_1^p(x)}(h_{\L}(x))+f^{k_1^q(z)}(\sigma_{\L_2}^{l_1^p(x)}(h_\L(x))) \}\\
 & -\{ f^{l_1^q(z)}(h_{\L}(z)) + f^{k_1^p(x)}(\sigma_{\L_2}^{l_1^q(z)}(h_{\L}(z)))\} \\ 
=&
\{f^{l_1^p(x)}(h_{\L}(x))+f^{k_1^q(z)}(\sigma_{\L_2}^{k_1^p(x)}(h_\L(\sigma_{\L_1}^p(x))))\}\\
 & -\{ f^{l_1^q(z)}(h_{\L}(z)) + f^{k_1^p(x)}(\sigma_{\L_2}^{k_1^q(z)}(h_{\L}(\sigma_{\L_1}^q(z))))\} \\ 
=&
f^{l_1^p(x)}(h_{\L}(x))
+\{ f^{k_1^q(z)+k_1^p(x)}(h_\L(\sigma_{\L_1}^p(x)))
- f^{k_1^p(x)}(h_\L(\sigma_{\L_1}^p(x)))\}\\
& - f^{l_1^q(z)}(h_{\L}(z)) 
-\{ f^{k_1^p(x)+k_1^q(z)}(h_{\L}(\sigma_{\L_1}^q(z)))
- f^{k_1^q(z)}(h_{\L}(\sigma_{\L_1}^q(z)))\} \\ 
=&
f^{l_1^p(x)}(h_{\L}(x))
- f^{k_1^p(x)}(h_\L(\sigma_{\L_1}^p(x))) - f^{l_1^q(z)}(h_{\L}(z)) 
+ f^{k_1^q(z)}(h_{\L}(\sigma_{\L_1}^q(z))) \\ 
=& \Psi_{h_\L}(f)^p(x) - \Psi_{h_\L}(f)^q(z) 
= \Psi_{h_\L}(f)_{G_{\L}}(x, p-q,z).
\end{align*}
\end{proof}
For $f \in C(X_\L,\Z)$ and the groupoid homomorphism
 $f_{G_{\frak L}}:G_{\frak L}\longrightarrow \Z$ defined in \eqref{eq:fGL},
we define one-parameter unitaries $U_t(f), t \in \T $ on $\ell^2(G_\L)$ by setting
\begin{equation}
[U_t(f)\xi](x, n,z) =
\exp{(2\pi\sqrt{-1}f_{G_{\frak L}}(x,n,z)t)} \xi(x,n,z),
\quad \xi \in \ell^2(G_\L),\, \,  t \in \T. 
\end{equation}
As in \cite{MaDynam2020},
the one-parameter unitaries satisfy the condition
$U_t(f) X U_t(f)^* \in\OL$ for $X \in \OL$,
the family $U_t(f), t \in \T$ gives rise to an action of $\T$ on the $C^*$-algebra
$\OL$ by 
$$
\rho^{\L,f}_t = \Ad(U_t(f)) \in \Aut(\OL), \qquad t \in \T.
$$
The automorphism $\rho^{\L,f}_t$ is often writen 
as $\rho^{\L,f_{G_\L}}_t$ for the groupoid homomorphism
$f_{G_\L}:G_\L\longrightarrow \Z$.
We then have the following proposition.
\begin{proposition}\label{prop:3.8}
Let $\L_1, \L_2$ be left-resolving $\lambda$-graph systems.
Assume that 
there exist a homeomorphism
$h_{\L}: X_{{\frak L}_1}\longrightarrow X_{{\frak L}_2}$
and continuous maps
$k_{i}, l_{i}: X_{\L_i}\longrightarrow \Z, i=1,2$
such that 
\begin{align*}
\sigma_{\L_2}^{k_{1}(x)}(h_\L(\sigma_{\L_1}(x)))
 =\sigma_{\L_2}^{l_{1}(x)}(h_\L(x)), \qquad x \in X_{\L_1}, \\
\sigma_{\L_1}^{k_{2}(y)}(h_\L^{-1}(\sigma_{\L_2}(y)))
 =\sigma_{\L_1}^{l_{2}(y)}(h_\L^{-1}(y)), \qquad y \in X_{\L_2}.
\end{align*}
Then there exists 
an isomorphism $\Phi: \mathcal{O}_{\L_1}\longrightarrow \mathcal{O}_{\L_2}$
of $C^*$-algebras such that 
$\Phi(\mathcal{D}_{\L_1}) = \mathcal{D}_{\L_2}$
with $\Phi(g) = g \circ h_\L^{-1}$ for $g \in \D_{\L_1}$
and
\begin{equation}\label{eq:3.20}
\Phi \circ \rho^{ {\L_1},{\Psi_{h_\L}(f)}}_t 
= \rho^{{\L_2}, f}_t\circ \Phi,
\qquad f \in C(X_{\L_2},\Z), \, t \in \T.
\end{equation}
\end{proposition}
\begin{proof}
By Lemma \ref{lem:3.1} and Lemma \ref{lem:3.4},
there exists a groupoid isomorphism
$\varphi:G_{\L_1}\longrightarrow G_{\L_2}$
satisfying  
\begin{equation} \label{eq:fglvarphi}
f_{G_\L}\circ\varphi = {\Phi_{h_\L}(f)}_{G_\L},
\qquad f \in C(X_{\L_2},\Z).
\end{equation}
By \cite[Proposition 3.3]{MaDynam2020},
there exists 
an isomorphism $\Phi: \mathcal{O}_{\L_1}\longrightarrow \mathcal{O}_{\L_2}$
of $C^*$-algebras such that 
$\Phi(\mathcal{D}_{\L_1}) = \mathcal{D}_{\L_2}$
with $\Phi(g) = g\circ h_{\L}^{-1}$ for $g \in \D_{\L_1}$
and
\begin{equation*}
\Phi \circ \rho^{ {\L_1},f_{G_\L} \circ \varphi}_t 
= \rho^{{\L_2}, f_{G_\L}}_t\circ \Phi,
\qquad f \in C(X_{\L_2},\Z), \, t \in \T.
\end{equation*}
By \eqref{eq:fglvarphi}, 
we obtain the equality \eqref{eq:3.20}.
\end{proof}
In \cite{MaDynam2020}, the notion of $(\L_1,\L_2)$-continuously orbit equivalence
between one-sided subshifts
$(X_{\Lambda_1},\sigma_{\Lambda_1})$ 
and 
$(X_{\Lambda_2},\sigma_{\Lambda_2})$ 
was introduced. 
The following notion of $(\L_1,\L_2)$-properly continuously orbit equivalence  
is stronger than it.
\begin{definition}[{cf. \cite[Definition 4.1]{MaDynam2020}, \cite[Section 6]{MaYMJ2010}}]
Let $\L_1, \L_2$ be left-resolving $\lambda$-graph systems
and
$\Lambda_1, \Lambda_2$ be their presenting subshifts, respectively.
The one-sided subshifts 
$(X_{\Lambda_1},\sigma_{\Lambda_1})$ 
and 
$(X_{\Lambda_2},\sigma_{\Lambda_2})$ 
are said to be 
$(\L_1,\L_2)$-{\it properly continuously orbit equivalent}\/
if 
there exist homeomorphisms
$h_{\L}: X_{{\frak L}_1}\longrightarrow X_{{\frak L}_2}$
and 
$h_{\Lambda}: X_{\Lambda_1}\longrightarrow X_{\Lambda_2}$
and
continuous maps
$k_{\Lambda_i}, l_{\Lambda_i}: X_{\Lambda_i}\longrightarrow \Z, i=1,2$
such that 
\begin{equation} \label{eq:pih}
\pi_{\L_2}\circ h_\L = h_{\Lambda}\circ \pi_{\L_1}
\end{equation}
and
\begin{align}
\sigma_{\L_2}^{k_{\Lambda_1}(\pi_{\L_1}(x))}(h_\L(\sigma_{\L_1}(x)))
 =\sigma_{\L_2}^{l_{\Lambda_1}(\pi_{\L_1}(x))}(h_\L(x)), 
\qquad x \in X_{\L_1}, \label{eq:pcoe1}\\
\sigma_{\L_1}^{k_{\Lambda_2}(\pi_{\L_2}(y))}(h_\L^{-1}(\sigma_{\L_2}(y)))
 =\sigma_{\L_1}^{l_{\Lambda_2}(\pi_{\L_2}(y))}(h_\L^{-1}(y)), 
\qquad y \in X_{\L_2},\label{eq:pcoe2}
\end{align}
where $\pi_{\L_i}:X_{\L_i}\longrightarrow X_{\Lambda_i}, i=1,2$ 
denotes
the surjection defined by \eqref{eq:pilambda}.
\end{definition}
By \eqref{eq:pih} and \eqref{eq:pcoe1},\eqref{eq:pcoe2},
the following lemma is straightforward. 
\begin{lemma}\label{lem:3.7}
Let $\L_1, \L_2$ be left-resolving $\lambda$-graph systems
and $\Lambda_1, \Lambda_2$ be their presenting subshifts, respectively.
Suppose that the one-sided subshifts 
$(X_{\Lambda_1},\sigma_{\Lambda_1})$ 
and 
$(X_{\Lambda_2},\sigma_{\Lambda_2})$ 
are $(\L_1,\L_2)$-properly continuously orbit equivalent.
\begin{enumerate}
\renewcommand{\theenumi}{\roman{enumi}}
\renewcommand{\labelenumi}{\textup{(\theenumi)}}
\item
The following equalities hold:
\begin{align*}
\sigma_{\Lambda_2}^{k_{\Lambda_1}(a)}(h_\Lambda(\sigma_{\Lambda_1}(a)))
 =\sigma_{\Lambda_2}^{l_{\Lambda_1}(a)}(h_\Lambda(a)), \qquad a \in X_{\Lambda_1}, \\
\sigma_{\Lambda_1}^{k_{\Lambda_2}(b)}(h_\Lambda^{-1}(\sigma_{\Lambda_2}(b)))
 =\sigma_{\Lambda_1}^{l_{\Lambda_2}(b)}(h_\Lambda^{-1}(b)), \qquad b \in X_{\Lambda_2}.
\end{align*}
\item
For $f \in C(X_{\Lambda_2},\Z)$, put 
\begin{equation} \label{eq:3.18}
\Psi_{h_\Lambda}(f)(a) 
= \sum_{i=0}^{l_{\Lambda_1}(a)}f(\sigma_{\Lambda_2}^i(h_\Lambda(a)))
- \sum_{j=0}^{k_{\Lambda_1}(a)}f(\sigma_{\Lambda_2}^j(h_\Lambda(\sigma_{\Lambda_1}(a)))),
\qquad a \in X_{\Lambda_1}.
\end{equation}
Then we have
\begin{equation}\label{eq:lem3.7ii}
\Psi_{h_\L}(f\circ\pi_{\L_2}) = \Psi_{h_\Lambda}(f)\circ\pi_{\L_1}, \qquad f \in C(X_{\Lambda_2},\Z).
\end{equation}
\end{enumerate}
\end{lemma}
%\begin{proof}
%\end{proof}
Let $\L$ be a left-resolving $\lambda$-graph system 
and $\Lambda$ be its presenting subshift.
For $f \in C(X_\Lambda,\Z),$ 
we have $f \circ \pi_\L \in C(X_{\L},\Z)$
so that the automorphisms
$\rho^{\L,f\circ\pi_\L}_t, t \in \T$ of $\OL$
are defined.
We denote them by $\rho^{\Lambda,f}_t, t \in \T$.
By Proposition \ref{prop:3.8} and
Lemma \ref{lem:3.7},
we may see the following lemma.
\begin{lemma}\label{lemma:3.8}
Let $\L_1$ and $ \L_2$ be left-resolving $\lambda$-graph systems.
If the presenting one-sided subshifts 
$(X_{\Lambda_1},\sigma_{\Lambda_1})$ 
and 
$(X_{\Lambda_2},\sigma_{\Lambda_2})$ 
are 
$(\L_1,\L_2)$-properly continuously orbit equivalent
via a homeomorphism $h_\Lambda: X_{\Lambda_1}\longrightarrow X_{\Lambda_2}$,
then there exists an isomorphism
$\Phi:\mathcal{O}_{\L_1}\longrightarrow \mathcal{O}_{\L_2}$
of $C^*$-algebras such that  
$\Phi(\mathcal{D}_{\Lambda_1}) = \mathcal{D}_{\Lambda_2}$
with $\Phi(g) = g \circ h_\Lambda^{-1}$ for $g \in \D_{\Lambda_1}$
and
\begin{equation}\label{eq:propcoe}
\Phi \circ \rho^{ {\Lambda_1},{\Psi_{h_\Lambda}(f)}}_t 
= \rho^{{\Lambda_2}, f}_t\circ \Phi,
\qquad f \in C(X_{\Lambda_2},\Z), \, t \in \T.
\end{equation}
\end{lemma}
\begin{proof}
Assume that the presenting one-sided subshifts 
$(X_{\Lambda_1},\sigma_{\Lambda_1})$ 
and 
$(X_{\Lambda_2},\sigma_{\Lambda_2})$ 
are $(\L_1,\L_2)$-properly continuously orbit equivalent
via a homeomorphism $h_{\Lambda}:X_{\Lambda_1}\longrightarrow X_{\Lambda_2}$.
Take a homeomorphism 
$h_{\L}: X_{{\frak L}_1}\longrightarrow X_{{\frak L}_2}$
satisfying
\eqref{eq:pih}
and
\eqref{eq:pcoe1}, \eqref{eq:pcoe2}.
As the embedding 
$\mathcal{D}_{\Lambda_i}\hookrightarrow \mathcal{D}_{\L_i}
$
is induced by the factor mp
$\pi_{\L_i}:X_{\L_i}\longrightarrow X_{\Lambda_i}, i=1,2$,
the isomorphism
$\Phi:\mathcal{O}_{\L_1}\longrightarrow \mathcal{O}_{\L_2}$
of $C^*$-algebras in Proposition \ref{prop:3.8} holding
$\Phi(g) = g\circ h_\L^{-1}, g \in \mathcal{D}_{\L_1}$
satisfies
$\Phi(\mathcal{D}_{\Lambda_1}) = \mathcal{D}_{\Lambda_2}$
because of \eqref{eq:pih}. 
%%%%%%%
%then we know 
%\begin{equation*}
%\Psi_{h_\Lambda}(f) \circ \pi_{\L_1} = \Psi_{h_\L}(f\circ \pi_{\L_2})
%\end{equation*}
%%%%%%%%%%%%%%%%%%
By the equality \eqref{eq:3.20} together with \eqref{eq:lem3.7ii}, 
we obtain the equality \eqref{eq:propcoe}.
\end{proof}
Now let $\Lambda$ be a normal subshift. 
Take the minimal $\lambda$-graph system 
$\LLmin$ of $\Lambda$ and consider the $C^*$-algebra
$\mathcal{O}_{\LLmin}$ that is written
$\OLmin$.
For  $f \in C(X_\Lambda,\Z)$,
the automorphisms
$\rho^{\Lambda,f}_t, t \in \T$
act on $\OLmin$
and yield an action written 
$\rho^{\Lambda,f}
$ of $\T$.
It is called the gauge action with potential $f$. 
%If in particular for $f\equiv 1$,
%the action $\rho^{\Lambda, 1}_t, t \in \T$
%on $\OLmin$ is called the gauge action on 
%$\OLmin$ and written $\rho^\Lambda_t, t\in \T.$
We identify the commutative $C^*$-subalgebras 
$\DL, \DLam$ with $C(X_\L), C(X_\Lambda)$
in a natural way, respectively.
We reach the following theorem
that proves Theorem \ref{thm:main} (i) $\Longrightarrow$ (ii).
\begin{theorem}\label{thm:main1}
Let $\Lambda_1$ and $\Lambda_2$ be normal subshifts.
If the one-sided subshifts 
$(X_{\Lambda_1},\sigma_{\Lambda_1})$ 
and 
$(X_{\Lambda_2},\sigma_{\Lambda_2})$ 
are topologically  conjugate,
then there exists an isomorphism
$\Phi:\OLamonemin\longrightarrow\OLamtwomin$ of $C^*$-algebras such that 
$\Phi({\mathcal{D}}_{\Lambda_1}) ={\mathcal{D}}_{\Lambda_2}$
and
\begin{equation}
\Phi\circ\rho^{\Lambda_1, f\circ h}_t = \rho^{\Lambda_2,f}_t\circ\Phi\qquad
\text{ for all }\quad f \in C(X_{\Lambda_2},\Z),\quad  t \in \T,
\end{equation}
where
$h:X_{\Lambda_1}\longrightarrow X_{\Lambda_2}$
is the homeomorphism satisfying
$\Phi(g) = g \circ h^{-1}, g \in C(X_{\Lambda_1})$
under the natural identification between $\D_{\Lambda_i}$ and $C(X_{\Lambda_i}), i=1,2.$
\end{theorem}
\begin{proof}
Assume that there exists a homeomorphism
$h:X_{\Lambda_1}\longrightarrow X_{\Lambda_2}$
such that $h\circ \sigma_{\Lambda_1} = \sigma_{\Lambda_2}\circ h$.
By \cite[Proposition 7.5]{MaPre2020a}, the one-sided subshifts
$(X_{\Lambda_1},\sigma_{\Lambda_1})$ 
and 
$(X_{\Lambda_2},\sigma_{\Lambda_2})$ 
are $(\L_{\Lambda_1}^\min,\L_{\Lambda_2}^\min)$-conjugate.
Hence by \cite[Definition 7.1]{MaPre2020a},
there exists a topological conjugacy
$h_\L: X_{\L_{\Lambda_1}^\min}\longrightarrow X_{\L_{\Lambda_2}^\min}$
such that 
$\pi_{\L_{\Lambda_2}^\min}\circ h_\L = h\circ \pi_{\L_{\Lambda_1}^\min}$,
so that it
yields an $(\L_{\Lambda_1}^\min,\L_{\Lambda_2}^\min)$-properly 
continuous orbit equivalence between
$(X_{\Lambda_1},\sigma_{\Lambda_1})$ 
and 
$(X_{\Lambda_2},\sigma_{\Lambda_2}).$ 
Lemma \ref{lemma:3.8} 
tells us 
that 
there exists an isomorphism
$\Phi:\mathcal{O}_{\L_{\Lambda_1}^\min}\longrightarrow \mathcal{O}_{\L_{\Lambda_2}^\min}$
of $C^*$-algebras
satisfying   
$\Phi(\mathcal{D}_{\Lambda_1}) = \mathcal{D}_{\Lambda_2}$
with $\Phi(g) = g \circ h^{-1}$ for $g \in \D_{\Lambda_1}$
and
\eqref{eq:propcoe}.
Now 
$h:X_{\Lambda_1}\longrightarrow X_{\Lambda_2}$
is a topological conjugacy, so that 
the equality
\begin{equation}
\Psi_{h}(f)  = f \circ h \qquad \text{ for all } f \in C(X_{\Lambda_2},\Z)
\label{eq:Psihlambda3.18}
\end{equation}
holds.
We thus have the desired assertion by the equality \eqref{eq:propcoe}
together with \eqref{eq:Psihlambda3.18}.
\end{proof}

%%%%%%%%%%%%%%%%%%%%%%%%%%%%%%%%%%%%%%%%%%%%%%%%%%%%%%%%%%%%%%%%%%%%%%%%%%%%%
%%%%%%%%%%%%%%%%%%%%%%%%%%%%%%%%%%%%%%%%%%%%%%%%%%%%%%%%%%%%%%%%%%%%%%%%%%%%%%%
\section{Proof of Theorem \ref{thm:main} (ii) $\Longrightarrow$ (i)} 
%%%%%%%%%%%%%%%%%%%%%%%%%%%%%%%%%%%%%%%%%%%%%%%%%%%%%%%%%%%%%%%%%%%%%%%%
%%%%%%%%%%%%%%%%%%%%%%%%%%%%%%%%%%%%%%%%%%%%%%%%%%%%%%%%%%%%%%%%%%%
%%%%%%%%%%%%%%%%%%%%%%%%%%%%%%%%%%%
In this section, 
we will give a proof of the implication of Theorem \ref{thm:main} (ii) $\Longrightarrow$ (i).
Let $\Lambda$ be a normal subshift.
Let us denote by $S_1,\dots, S_N$ the canonical generating partial isometries of the $C^*$-algebra
$\OLmin$.
Recall that the $C^*$-subalgebras 
$\DLam$ and  $\DLmin$ of $\OLmin$ are defined in \eqref{eq:DL} and \eqref{eq:DLam}, respectively such that 
$\DLam \subset \DLmin$.
\begin{lemma}\label{lem:alpha}
Let $\alpha$ be an automorphism of 
$\OLmin$ such that its restriction to the subalgebra $\DLam$ 
is the identity. 
Then we have
\begin{enumerate}
\renewcommand{\theenumi}{\roman{enumi}}
\renewcommand{\labelenumi}{\textup{(\theenumi)}}
\item
$\alpha(S_\mu^*)S_\nu \in \DLmin$ for all $\mu, \nu \in B_*(\Lambda)$ with $|\mu| = |\nu|$.
\item
Put $\lambda_{\frak L} (X) = \sum_{j=1}^N S_j^* X S_j, X \in \DLmin$.
Then we have $\alpha\circ\lambda_{\frak L}= \lambda_{\frak L}\circ\alpha.$ 
\item The restriction of $\alpha$ to the subalgebra $\DLmin$
 is the identity.
\end{enumerate}
\end{lemma}
\begin{proof}
(i) Let $\mu, \nu \in B_*(\Lambda)$ satisfy
$|\mu| = |\nu|$.
If $\mu \ne \nu$, we see that 
$$
\alpha(S_\mu^*)S_\nu 
= \alpha(S_\mu^*)S_\nu S_\nu^*S_\nu 
= \alpha(S_\mu^*S_\nu S_\nu^*) S_\nu 
=0.
$$
Suppose that 
$\mu = \nu$.
For $\xi \in B_*(\Lambda)$, we have
\begin{equation*}
\alpha(S_\mu^*)S_\mu S_\xi S_\xi^*
=  \alpha(S_\mu^*) \alpha(S_{\mu\xi} S_{\mu\xi}^*) S_\mu 
=  \alpha(S_\mu^* S_\mu S_\xi S_\xi^*S_\mu^*) S_\mu 
%= & \alpha( S_\xi S_\xi^*S_\mu^* S_\mu S_\mu^*) S_\mu \\
=   S_\xi S_\xi^* \alpha(S_\mu^*) S_\mu
\end{equation*}
so that $\alpha(S_\mu^*)S_\mu$ commutes with $ S_\xi S_\xi^*$ for all $\xi \in B_*(\Lambda).$
Hence $\alpha(S_\mu^*)S_\mu$ belongs to 
$\DLam^\prime \cap \OLmin$ that is 
$\DLmin$ by Proposition \ref{prop:2.4}.

(ii) 
We first note that 
the condition $\alpha(\DLam) = \DLam$ implies 
$\alpha(\DLam)^\prime \cap \OLmin =\DLam^\prime \cap \OLmin$
so that 
$\alpha(\DLmin) =\DLmin$
by Proposition \ref{prop:2.4}.
For $\mu = (\mu_1,\mu_2,\dots,\mu_m)$, put
$\bar{\mu} = (\mu_2,\dots,\mu_m)$ so that $\mu = \mu_1\bar{\mu}$.
We then have  for $i=1,2,\dots,m(l)$,
\begin{equation*}
\alpha\circ\lambda_{\frak L}( S_\mu E_i^l S_\mu^*)
%= &\alpha(\sum_{j=1}^N S_j^* S_\mu E_i^l S_\mu^* S_j) \\
= 
\alpha(S_{\mu_1}^*S_{\mu_1} S_{\bar{\mu}} E_i^l S_{\bar{\mu}}^*S_{\mu_1}^*S_{\mu_1} )
= 
\alpha(S_{\bar{\mu}} E_i^l S_{\bar{\mu}}^*)\alpha(S_{\mu_1}^*S_{\mu_1} ).
\end{equation*}
On the other hand, we have by (i)
\begin{align*}
\lambda_{\frak L}\circ\alpha( S_\mu E_i^l S_\mu^*)
= &\sum_{j=1}^N S_j^* \alpha( S_{\mu_1}) \alpha( S_{\bar{\mu}} E_i^l S_{\bar{\mu}}^*) \alpha(S_{\mu_1}^*)S_j \\
= & S_{\mu_1}^* \alpha( S_{\mu_1}) \alpha( S_{\bar{\mu}} E_i^l S_{\bar{\mu}}^*) \alpha(S_{\mu_1}^*)S_{\mu_1} \\
= &  \alpha( S_{\bar{\mu}} E_i^l S_{\bar{\mu}}^*) \alpha(S_{\mu_1}^*)S_{\mu_1}S_{\mu_1}^* \alpha( S_{\mu_1}) \\
= &  \alpha( S_{\bar{\mu}} E_i^l S_{\bar{\mu}}^*) \alpha(S_{\mu_1}^*S_{\mu_1}) 
\end{align*}
so that we have 
$\alpha\circ\lambda_{\frak L}( S_\mu E_i^l S_\mu^*)
=
\lambda_{\frak L}\circ\alpha( S_\mu E_i^l S_\mu^*) 
$
and hence
$\alpha\circ\lambda_{\frak L} =\lambda_{\frak L}\circ\alpha.$

(iii)
For $\mu \in B_k(\Lambda)$, we have
$S_\mu^* S_\mu = \lambda_{\frak L}^k(S_\mu S_\mu^*)$
so that 
\begin{equation}
\alpha(S_\mu^* S_\mu) = \alpha(\lambda_{\frak L}^k(S_\mu S_\mu^*))
=\lambda_{\frak L}^k( \alpha(S_\mu S_\mu^*))
=\lambda_{\frak L}^k(S_\mu S_\mu^*)=S_\mu^* S_\mu. \label{eq:alphaSmu}
\end{equation}
Since 
$E_i^l$ is written by some products among $S_\mu^*S_\mu, 1-S_\mu^*S_\mu$
for $\mu \in B_l(\Lambda)$ as in  \eqref{eq:Eil},
we know
$\alpha(E_i^l) = E_i^l.$
For $\mu \in \Gamma_l^-(v_i^l)$, 
as $S_\mu^*\alpha(S_\mu)$ commutes with 
$E_i^l$, we have
\begin{equation*}
\alpha(S_\mu E_i^l S_\mu^*)
% = \alpha(S_\mu S_\mu^*)\alpha(S_\mu E_i^l S_\mu^*) 
 = S_\mu S_\mu^*\alpha(S_\mu E_i^l S_\mu^*) 
 = S_\mu\cdot S_\mu^*\alpha(S_\mu) \cdot  E_i^l \alpha(S_\mu^*) 
 = S_\mu  E_i^l S_\mu^*\alpha(S_\mu) \alpha(S_\mu^*) 
 = S_\mu  E_i^l S_\mu^*
\end{equation*}
so that we see
$\alpha(S_\mu E_i^l S_\mu^*) = S_\mu E_i^l S_\mu^*.$
\end{proof}
We thus have the following lemma.
\begin{lemma}\label{lem:ualpha}
Let $\alpha$ be an automorphism of 
$\OLmin$ such that its restriction to the subalgebra $\DLam$ 
is the identity. 
Then there exists a unitary $U_\alpha \in \DLmin$
such that 
\begin{equation}
\alpha(S_i) = U_\alpha S_i, \qquad i=1,2,\dots,N.
\end{equation}
\end{lemma}
\begin{proof}
We set $U_\alpha = \sum_{j=1}^N \alpha(S_j) S_j^*.$
By Lemma \ref{lem:alpha} (iii) together with \eqref{eq:alphaSmu}, 
we know $\alpha(S_i^* S_i) = S_i^* S_i$.
It then follows that 
\begin{equation*}
U_\alpha S_i =\sum_{j=1}^N \alpha(S_j) S_j^*S_i = \alpha(S_i) S_i^*S_i=\alpha(S_i).
\end{equation*}
We also have 
\begin{equation*}
U_\alpha U_\alpha^*
%=(\sum_{i=1}^N \alpha(S_i) S_i^*)(\sum_{j=1}^N \alpha(S_j) S_j^*)^*
=\sum_{i=1}^N \alpha(S_i) S_i^*S_i\alpha(S_i^*) 
=\sum_{i=1}^N \alpha(S_iS_i^*) =1 
\end{equation*}
and
\begin{equation*}
 U_\alpha^* U_\alpha
%=(\sum_{i=1}^N \alpha(S_i) S_i^*)^*(\sum_{j=1}^N \alpha(S_j) S_j^*)
=\sum_{i=1}^N S_i \alpha(S_i^*S_i) S_i^* 
=\sum_{i=1}^N S_i S_i^* =1. 
\end{equation*}
We will next show that $U_\alpha$ belongs to
$\DLmin.$
For $\mu =(\mu_1,\dots,\mu_m) \in B_m(\Lambda)$ and
$\bar{\mu} =(\mu_2,\dots,\mu_m) \in B_{m-1}(\Lambda)$,
we have
\begin{align*}
U_\alpha S_\mu S_\mu^* U_\alpha
= & \sum_{j,k=1}^N \alpha(S_j) 
S_j^* S_{\mu_1}S_{\bar{\mu}}S_{\bar{\mu}}^*S_{\mu_1}^*S_k \alpha(S_k^*)\\
= &  \alpha(S_{\mu_1}) 
S_{\mu_1}^* S_{\mu_1}S_{\bar{\mu}}S_{\bar{\mu}}^*S_{\mu_1}^*S_{\mu_1} \alpha(S_{\mu_1}^*)\\
= &  \alpha(S_{\mu_1}S_{\mu_1}^* S_{\mu_1}) \alpha(S_{\bar{\mu}}S_{\bar{\mu}}^*)
\alpha(S_{\mu_1}^*S_{\mu_1}) \alpha(S_{\mu_1}^*)\\
= &  \alpha(S_{\mu}S_{\mu}^*) =S_\mu S_\mu^* \\
\end{align*}
so that $U_\alpha$ commutes with $ S_\mu S_\mu^*, \mu \in B_*(\Lambda)$.
This shows that 
$U_\alpha$ belongs to
$\DLmin$ by Proposition \ref{prop:2.4}.
\end{proof}
We will prove the following theorem 
that is the converse of Theorem \ref{thm:main1}.
\begin{theorem}\label{thn:main2}
Let $\Lambda_1, \Lambda_2$ be normal subshifts.
If there exists an isomorphism 
$\Phi:\OLamonemin\longrightarrow\OLamtwomin$ of $C^*$-algebras such that 
$\Phi({\mathcal{D}}_{\Lambda_1}) ={\mathcal{D}}_{\Lambda_2}$
and
\begin{equation} \label{eq:gauge}
\Phi\circ\rho^{\Lambda_1, f\circ h_\Lambda}_t = \rho^{\Lambda_2,f}_t\circ\Phi\qquad
\text{ for all }\quad f \in C(X_{\Lambda_2},\Z),\quad  t \in \T,
\end{equation}
where $h_\Lambda: X_{\Lambda_1}\longrightarrow X_{\Lambda_2}$
is a homeomorphism satisfying 
$\Phi(g) = g\circ h_\Lambda^{-1}$ for $g \in \D_{\Lambda_1},$   
then  $h_\Lambda: X_{\Lambda_1}\longrightarrow X_{\Lambda_2}$
gives rise to a topological conjugacy between 
$(X_{\Lambda_1},\sigma_{\Lambda_1})$ 
and 
$(X_{\Lambda_2},\sigma_{\Lambda_2}).$
\end{theorem}
\begin{proof}
Assume that 
 there exists an isomorphism 
$\Phi:\OLamonemin\longrightarrow\OLamtwomin$ of $C^*$-algebras satisfying  
$\Phi({\mathcal{D}}_{\Lambda_1}) ={\mathcal{D}}_{\Lambda_2}$
and the equality \eqref{eq:gauge}.
%\begin{equation} \label{eq:gauge}
%\Phi\circ\rho^{\Lambda_1, f\circ h_\Lambda}_t 
%= \rho^{\Lambda_2,f}_t\circ\Phi\qquad
%\text{ for all }\quad f \in C(X_{\Lambda_2},\Z),\quad  t \in \T.
%\end{equation}
In particular for the constant function $f\equiv 1$ on $X_{\Lambda_2},$
the equality \eqref{eq:gauge} goes to 
\begin{equation}
\Phi\circ\rho^{\Lambda_1}_t = \rho^{\Lambda_2}_t\circ\Phi\qquad
\text{ for all }\quad  t \in \T.
\end{equation}
By \cite[Proposition 8.9]{MaPre2020a}, we know that 
$(X_{\Lambda_1},\sigma_{\Lambda_1})$ 
and 
$(X_{\Lambda_2},\sigma_{\Lambda_2})$
are  $({\frak L}_1^{\min},{\frak L}_2^{\min})$-eventually conjugate.
Hence it is  
$({\frak L}_1^{\min},{\frak L}_2^{\min})$-properly continuously orbit equivalent
via a homeomorphism $h_\Lambda:X_{\Lambda_1}\longrightarrow X_{\Lambda_2}$,
so that by Lemma \ref{lemma:3.8},
there exists an isomorphism 
$\Phi_1:\OLamonemin\longrightarrow\OLamtwomin$ 
of $C^*$-algebras such that 
$\Phi_1({\mathcal{D}}_{\Lambda_1}) ={\mathcal{D}}_{\Lambda_2}$
with
$\Phi_1(a) = a\circ h_\Lambda^{-1}, a \in \D_{\Lambda_1}$
and
\begin{equation} \label{eq:phi1coegauge}
\Phi_1\circ\rho^{\Lambda_1, \Psi_{h_\Lambda}(f)}_t = \rho^{\Lambda_2,f}_t\circ\Phi_1\qquad
\text{ for all }\quad f \in C(X_{\Lambda_2},\Z),\quad  t \in \T.
\end{equation}
Since the original isomorphism
 $\Phi:\OLamonemin\longrightarrow\OLamtwomin$ 
 of $C^*$-algebras satisfying  
$\Phi({\mathcal{D}}_{\Lambda_1}) ={\mathcal{D}}_{\Lambda_2}$
and
$\Phi(a) = a\circ h_\Lambda^{-1}$ for $a \in \D_{\Lambda_1}$,
we know that 
$\Phi^{-1}\circ \Phi_1$ is the identity on ${\mathcal{D}}_{\Lambda_1}$.
This means that 
$\Phi^{-1}\circ \Phi_1$ is an automorphism of
$\OLamonemin$ such that its restriction to the subalgebra 
${\mathcal{D}}_{\Lambda_1}$ is the identity.
Let us denote by $S_1,\dots, S_N$ the canonical generating partial isometries
of the $C^*$-algebra $\OLamonemin$.
By Lemma \ref{lem:ualpha}, there exists a unitary 
$U \in \mathcal{D}_{{\frak L}_{\Lambda_1}^{\min}} 
$ such that 
$\Phi^{-1}\circ \Phi_1(S_i) = U S_i, i=1,2,\dots, N.$
By putting $U_1 = \Phi(U)$, we have 
$$
\Phi_1(S_i) = U_1 \Phi(S_i), \qquad i=1,2,\dots,N.
$$
By \eqref{eq:phi1coegauge} and the equality
\begin{equation}
\rho^{\Lambda_1, \Psi_{h_\Lambda}(f)}_t(S_i) 
= \exp{(2 \pi \sqrt{-1}\Psi_{h_\Lambda}(f) t)} S_i, \qquad i=1,2,\dots,N,
\end{equation}
we have 
\begin{equation*}
\Phi_1(\exp{(2 \pi \sqrt{-1}\Psi_{h_\Lambda}(f) t)} S_i)
=\rho^{\Lambda_2, f}_t(U_1 \Phi(S_i)),  \qquad i=1,2,\dots,N
\end{equation*} 
and hence 
\begin{equation*}
\Phi_1(\exp{(2 \pi \sqrt{-1}\Psi_{h_\Lambda}(f) t)}) U_1\Phi( S_i)
=U_1 \rho^{\Lambda_2, f}_t( \Phi(S_i)),  \qquad i=1,2,\dots,N.
\end{equation*}
As the restriction of
$\Phi^{-1}\circ \Phi_1$ to $\mathcal{D}_{\Lambda_1}$
is the identity, we have 
\begin{equation*}
\Phi_1(\exp{(2 \pi \sqrt{-1}\Psi_{h_\Lambda}(f) t)})
=\Phi(\exp{(2 \pi \sqrt{-1}\Psi_{h_\Lambda}(f) t)})
\end{equation*}
so that 
\begin{equation*}
\Phi(\exp{(2 \pi \sqrt{-1}\Psi_{h_\Lambda}(f) t)}) U_1\Phi( S_i)
=U_1 \rho^{\Lambda_2, f}_t( \Phi(S_i)),  \qquad i=1,2,\dots,N.
\end{equation*}
Therefore we have  
\begin{equation*} 
\Phi\circ\rho^{\Lambda_1, \Psi_{h_\Lambda}(f)}_t(S_i) 
= \rho^{\Lambda_2,f}_t\circ\Phi(S_i),
 \qquad
%\text{ for all }\quad f \in C(X_{\Lambda_2},\Z),\quad  
t \in \T
\end{equation*}
and hence 
\begin{equation} \label{eq:phicoegauge}
\Phi\circ\rho^{\Lambda_1, \Psi_{h_\Lambda}(f)}_t 
= \rho^{\Lambda_2,f}_t\circ\Phi, \qquad
%\text{ for all }\quad f \in C(X_{\Lambda_2},\Z),\quad  
t \in \T.
\end{equation}
Since
\begin{equation*} 
\Phi\circ\rho^{\Lambda_1, f\circ h}_t 
= \rho^{\Lambda_2,f}_t\circ\Phi, \qquad
%\text{ for all }\quad f \in C(X_{\Lambda_2},\Z),\quad  
t \in \T,
\end{equation*}
we obtain that 

\begin{equation}
\Psi_{h_\Lambda}(f) = f \circ h_\Lambda
\qquad
\text{ for all }\quad f \in C(X_{\Lambda_2},\Z).
\end{equation}
We will reach the desired assertion from the following lemma.
\end{proof}
\begin{lemma}
Suppose that  
$(X_{\Lambda_1},\sigma_{\Lambda_1})$ and
$(X_{\Lambda_2},\sigma_{\Lambda_2})$
are
$({\frak L}_1^{\min},{\frak L}_2^{\min})$-properly continuously orbit equivalent
such that 
homeomorphisms
$h_{\L}: X_{{\frak L}_1}\longrightarrow X_{{\frak L}_2}$
and 
$h_{\Lambda}: X_{\Lambda_1}\longrightarrow X_{\Lambda_2}$
satisfy the conditions 
\eqref{eq:pih} and \eqref{eq:pcoe1}, \eqref{eq:pcoe2}.
Let
$
\Psi_{h_\Lambda}:C(X_{\Lambda_2},\Z) \longrightarrow C(X_{\Lambda_1},\Z)
$
be the map defined by \eqref{eq:3.18}.
%\begin{equation}\label{eq:4.10}
%\Psi_{h_\Lambda}(a) = 
%\sum_{i=0}^{l_{\Lambda_1}(a)} f(\sigma_{\Lambda_2}^i(h_\Lambda(a))) -
%\sum_{j=0}^{k_{\Lambda_1}(a)} f(\sigma_{\Lambda_2}^j(h_\Lambda(\sigma_{\Lambda_1}(a)))),
%\qquad f \in C(X_{\Lambda_1},\Z), \quad a \in X_{\Lambda_1}.
%\end{equation}
If $\Psi_{h_\Lambda}(f) = f\circ h_{\Lambda}$ 
for all $f \in C(X_{\Lambda_2},\Z)$,
then the homeomorphism
$h_\Lambda:X_{\Lambda_1}\longrightarrow X_{\Lambda_2}
$
gives rise to a topological conjugacy
between $(X_{\Lambda_1},\sigma_{\Lambda_1})$ and
$(X_{\Lambda_2},\sigma_{\Lambda_2})$.
\end{lemma}
\begin{proof}
We fix $a \in X_{\Lambda_1}$. 
By \eqref{eq:3.18} with the condition 
$\Psi_{h_\Lambda}(f)(a) = f(h_{\Lambda}(a))$
and the equality
$\sigma_{\Lambda_2}^{l_{\Lambda_1}(a)}(h_\Lambda(a))
=
\sigma_{\Lambda_2}^{k_{\Lambda_1}(a)}(h_\Lambda(\sigma_{\Lambda_1}(a))),
$
we have 
\begin{equation}\label{eq:star1}
\sum_{i=1}^{l_{\Lambda_1}(a)-1} f(\sigma_{\Lambda_2}^i(h_\Lambda(a)))
=
\sum_{j=0}^{k_{\Lambda_1}(a)-1} f(\sigma_{\Lambda_2}^j(h_\Lambda(\sigma_{\Lambda_1}(a)))),
\qquad f \in C(X_{\Lambda_2},\Z).
\end{equation}
Put
\begin{align*}
B  = & \{ \sigma_{\Lambda_2}^i(h_\Lambda(a)) \mid i=1, 2,\dots, l_{\Lambda_1}(a)-1\},\\
D = & \{ \sigma_{\Lambda_2}^j(h_\Lambda(\sigma_{\Lambda_1}(a))) \mid j=0,1, \dots, k_{\Lambda_1}(a)-1\}.
\end{align*}
By \eqref{eq:star1},we have
\begin{equation}\label{eq:star2}
\sum_{b \in B}f(b) = \sum_{d \in D}f(d)
\qquad \text{ for all } f \in C(X_{\Lambda_2},\Z).
\end{equation}
Suppose that $B\ne D$ as sets, 
and that there exists  an element $b_1 \in B$
such that $b_1 \not\in D$.
One may find a continuous function $f\in  C(X_{\Lambda_2},\Z)$ such that 
$$
f(b_1) =1, \qquad f(b) =0 \quad \text{ for all } b\in D\cup B\backslash\{b_1\}
$$
a contradiction to \eqref{eq:star2}.
Hence we see that $B \subset D$
and similarly $D \subset B$
so that $B = D$.
By \eqref{eq:star1}, we have
$l_{\Lambda_1}(a)= k_{\Lambda_1}(a)-1$
and there exist nonnegative integers
$p, q$ with $0\le p \le k_{\Lambda_1}(a)-1$
and $1\le q \le l_{\Lambda_1}(a)-1$
such that
\begin{equation}\label{eq:pq}
\sigma_{\Lambda_2}(h_\Lambda(a)) = \sigma_{\Lambda_2}^p(h_\Lambda(\sigma_{\Lambda_1}(a))), \qquad
h_\Lambda(\sigma_{\Lambda_1}(a)) = \sigma_{\Lambda_2}^q(h_\Lambda(a)).
\end{equation}
In case that $p=0$, we have 
$\sigma_{\Lambda_2}(h_\Lambda(a)) = h_\Lambda(\sigma_{\Lambda_1}(a)).$
Suppose that $p\ne 0$ so that $p\ge 1$.
By \eqref{eq:pq}, we have
\begin{equation}
\sigma_{\Lambda_2}(h_\Lambda(a)) = \sigma_{\Lambda_2}^p(h_\Lambda(\sigma_{\Lambda_1}(a)))
= \sigma_{\Lambda_2}^p(\sigma_{\Lambda_2}^q(h_\Lambda(a)))
= \sigma_{\Lambda_2}^{p+q}(h_\Lambda(a)).
\end{equation}
Since $p+q \ge 2$, the point $h_{\Lambda}(a)$ must be an eventually periodic point.
Hence we conclude that if a point $b =h_{\Lambda}(a)$ is not eventually periodic,
the equality $\sigma_{\Lambda_2}(h_\Lambda(a)) = h_\Lambda(\sigma_{\Lambda_1}(a))$
holds.
%%%
%Now the subshifts $\Lambda_i, i=1,2$ are both normal, 
%so that the sets of eventually periodic points  are dense, respectively.  
%%%%%%%%%%
As in the proof of \cite[Proposition 3.5]{MMETDS},
the homeomorphism
$h_\Lambda: X_{\Lambda_1}\longrightarrow X_{\Lambda_2}$ 
satisfying the condition
$$
\sigma_{\Lambda_2}^{k_{\Lambda_1}(a)}(h_\Lambda(\sigma_{\Lambda_1}(a)))
=\sigma_{\Lambda_2}^{l_{\Lambda_1}(a)}(h_\Lambda(a)),
\qquad a \in X_{\Lambda_1} 
$$
preserves the set of eventually periodic points
in $X_{\Lambda_i}$.
Since the set of non-eventually periodic points is dense in a normal subshift,
 we conclude that the equality
$\sigma_{\Lambda_2}(h_\Lambda(a)) = h_\Lambda(\sigma_{\Lambda_1}(a))$
holds for all points $a \in X_{\Lambda_1}$, so that 
the homeomorphism
$h_\Lambda: X_{\Lambda_1}\longrightarrow X_{\Lambda_2}$ 
is a topological conjugacy.
\end{proof}

%%%%%%%%%%%%%%%%%%%%%%%%%%%%%%%%%%%%%%%%%%%%%%%
\section{Example: Irreducible sofic shifts}
%%%%%%%%%%%%%%%%%%%%%%%%%%%%%%%%%%%%%%%%%%%%%%%%%%%
Let $G=(V, E, \lambda)$ be a finite directed labeled graph over finite alphabet $\Sigma$,
where
$V$ is a finite vertex set, $E$ is a finite edge set and
$\lambda: E\longrightarrow \Sigma$ is a labeling map.
The shift space $\Lambda_G$ of a  sofic shift $(\Lambda_G,\sigma_G)$  for the directed labeled graph $G$ is 
defined by 
\begin{equation*}
\Lambda_G =\{
(\lambda(e_n))_{n\in \Z} \in \Sigma^\Z \mid t(e_n) = s(e_{n+1}),\, e_n, e_{n+1} \in E, \,  n \in \Z\},
\end{equation*}
where
$t(e_n)$ denotes the terminal vertex of $e_n$ and
$s(e_{n+1})$ denotes the source vertex of $e_{n+1}.$
The sofic shift $(\Lambda_G,\sigma_G)$ is written $\Lambda_G$ or $\Lambda$ without specifying $G$ for brevity.
The one-sided sofic shift for $\Lambda_G$ is denoted by $X_{\Lambda_G}$ or $X_\Lambda$.
If $\Lambda_G$ is not any finite set, it is said to be nontrivial.
The class of irreducible nontrivial sofic shifts contains
that of irreducible topological Markov shifts defined by 
irreducible non-permutation matrices with entreis in $\{0,1\}$. 

Let $\Lambda$ be an irreducible nontrivial sofic shift over alphabet $\Sigma$.
There are many finite labeled graph that present the sofic shift.
Among them, there is a left-resolving, predecessor-separated irreducible
finite labeled graph written
$G_\Lambda^F = (V_\Lambda^F, E_\Lambda^F, \lambda_\Lambda^F)$.
It is a unique labeled graph called the (left) Fischer cover graph (cf. \cite{LM}). 
Let us denote by 
$\{ v_1,\dots,v_N\}$
the vertex set
$V_\Lambda^F$.
For $i,j = 1,\dots, N$ and $\alpha \in \Sigma$, 
consider the matrix:
\begin{equation*} %\label{eq:Aialphaj}
A(i,\alpha,j) 
= 
\begin{cases}
1 &  \text{ if } s(e) = v_i,  t(e) = v_j, \lambda_\Lambda^F(e) = \alpha
\text{ for some } e \in E_\Lambda^F,  \\
0 & \text{ otherwise. }
\end{cases}
\end{equation*}
Put
\begin{equation*}
\widehat{\Sigma} 
=\{(\alpha,i) \in \Sigma \times \{1,2,\dots,N\} \mid
 \lambda_\Lambda^F(e)= \alpha, t(e) = v_i \text{ for some } e \in E_\Lambda^F \}.
\end{equation*}
Define a matrix 
$\widehat{A}$ over $\widehat{\Sigma}$ 
by setting
\begin{equation}
\widehat{A}((\alpha,i), (\beta,j)) 
=
\sum_{k=1}^N  A(k,\alpha,i) A(i,\beta,j)
\quad
\text{ for } 
(\alpha,i), (\beta,j) \in \widehat{\Sigma}. \label{eq:5.1}
\end{equation}
As the labeled graph $G_\Lambda^F$ is left-resolving,
the $(\alpha,i), (\beta,j)$-entry 
$\widehat{A}((\alpha,i), (\beta,j))$
of the matrix $\widehat{A}$ is zero or one.
Since $\Lambda$ is irreducible, the matrix $\widehat{A}$ is an irreducible matrix.
Let us denote by $X_{\widehat{A}}$ the shift space of the topological Markov shift
$(X_{\widehat{A}}, \sigma_{\widehat{A}})$
defined by the matrix $\widehat{A}$.
There exists a  factor map
$\pi_\Lambda: X_{\widehat{A}} \longrightarrow X_\Lambda$
defined by
$\pi_\Lambda((\alpha_n,i_n)_{n \in \N}) =(\alpha_n)_{n\in \N}.$
It satisfies
$\pi_\Lambda\circ \sigma_{\widehat{A}} =\sigma_\Lambda\circ \pi_\Lambda$.  
Let $S_\alpha, \alpha \in \Sigma$ and
$E_i, i=1,2,\dots,N$ be partial isometries and projections
respectively 
satisfying the operator relations: 
\begin{equation}
\sum_{j=1}^{N} E_j = \sum_{\beta \in \Sigma}S_\beta S_\beta^* = 1, 
\qquad S_\alpha S_\alpha^* E_i = E_i S_\alpha S_\alpha^*, 
\qquad
S_\alpha^* E_i S_\alpha = \sum_{j=1}^{N} A(i,\alpha, j) E_j \label{eq:5.2}
\end{equation}
for $\alpha \in \Sigma,\ i=1, 2,\dots,N$.
Let us denote by ${\mathcal{O}}_{G^F_\Lambda}$
the universal $C^*$-algebra generated by 
$S_\alpha, \alpha \in \Sigma$ and
$E_i, i=1, 2, \dots,N$ satisfying the relations \eqref{eq:5.2}.
By the relations \eqref{eq:5.2},
we know that the $C^*$-algebra $\OLmin$  associated with  the sofic shift
$\Lambda$ is canonically isomorphic to the $C^*$-algebra 
${\mathcal{O}}_{G^F_\Lambda}$.
Define partial isometries
$$
S_{(\alpha,i)} = S_\alpha E_i \qquad \text{ for }\quad (\alpha,i) \in \widehat{\Sigma}.
$$
As in  \cite[Section 4]{MaPre2020a}, the $C^*$-algebra
$C^*(S_{(\alpha,i)}: (\alpha,i) \in \widehat{\Sigma})$ generated by 
$
S_{(\alpha,i)}, (\alpha,i) \in \widehat{\Sigma} 
$
is canonically isomorphic to 
the Cuntz-Krieger algebra 
${\mathcal{O}}_{\widehat{A}}$
 for the matrix $\widehat{A}$,
 that is simple and purely infinite. 
  Since
 \begin{equation*}
S_\alpha = \sum_{i=1}^N S_\alpha E_i = \sum_{i=1}^N S_{(\alpha,i)}, \qquad
E_i= \sum_{(\beta,j)\in \widehat{\Sigma}} A(i,\beta,j) S_{(\beta,j)} S_{(\beta,j)}^*,
\end{equation*}
we know that
$C^*(S_{(\alpha,i)}; (\alpha,i) \in \widehat{\Sigma})$
coincides with
${\mathcal{O}}_{G^F_\Lambda} $
and hence with $\OLmin$.
Define the commutative $C^*$-subalgebra $\mathcal{D}_{G_\Lambda^F}$ of 
$\mathcal{O}_{G_\Lambda^F}$ 
by setting
\begin{equation*}
\mathcal{D}_{G_\Lambda^F} = C^*(S_\mu E_i S_\mu^* \mid \mu \in B_*(\Lambda), i=1,2,\dots,N).
\end{equation*}
Through the identification between
$\OLmin$ and $\mathcal{O}_{G_\Lambda^F}$ ,
and between 
$\mathcal{O}_{G_\Lambda^F}$  and  ${\mathcal{O}}_{\widehat{A}}$,
we know that 
$\DLmin = \mathcal{D}_{G_\Lambda^F}$,
and
$\mathcal{D}_{G_\Lambda^F}= \mathcal{D}_{\widehat{A}}$,
where
$\mathcal{D}_{\widehat{A}} =C(X_{\widehat{A}})$
the canonical diagonal algebra of the Cuntz--Krieger algebra   
 $\mathcal{O}_{\widehat{A}}$.
Therefore these three pairs
\begin{equation*}
(\OLmin,\DLmin), \qquad
(\mathcal{O}_{G_\Lambda^F}, \mathcal{D}_{G_\Lambda^F}),\qquad
(\mathcal{O}_{\widehat{A}}, \mathcal{D}_{\widehat{A}})
\end{equation*}
are all identified with each other in natural way.
Under the identification between 
$\DLmin$ and $ \mathcal{D}_{G_\Lambda^F}$,
 the commutative $C^*$-subalgebra $\mathcal{D}_\Lambda =C(X_\Lambda)$
 of $\DLmin$  is regarded as a $C^*$-subalgebra of  
 $\mathcal{D}_{\widehat{A}}$ that comes from the factor map
 $\pi_\Lambda:X_{\widehat{A}} \longrightarrow X_\Lambda$ from the Fischer cover
 $X_{\widehat{A}}$ to $X_\Lambda$.
 Under the identification between 
 $\OLmin$ and $\mathcal{O}_{\widehat{A}}$,
 the gauge action $\rho^{\Lambda,f}$ on $\OLmin$ with potential $f \in C(X_\Lambda,\Z)$
 is identified with the gauge action
 $\rho^{\widehat{A},f}$ 
 on $\mathcal{O}_{\widehat{A}}$ with potential $f$. 
 Therefore we have the following proposition.
\begin{proposition}[{cf.  \cite[Proposition 4.2]{MaPre2020a}}]  \label{prop:5.1}
Let $\Lambda$ be a nontrivila  irreducible sofic shift.
Let $\widehat{A}$ be the transition matrix of the left Fischer cover 
graph of $\Lambda$ defined by \eqref{eq:5.1}.
Then there exists an isomorphism 
$\Phi_\Lambda:\OLmin\longrightarrow 
{\mathcal{O}}_{\widehat{A}}$
from the $C^*$-algebra 
$\OLmin$ associated with the normal subshift $\Lambda$
onto the simple purely infinite Cuntz--Krieger algebra
${\mathcal{O}}_{\widehat{A}}$
for the matrix $\widehat{A}$ such that 
\begin{gather*}
\Phi_\Lambda(\DLmin) = {\mathcal{D}}_{\widehat{A}} (= C(X_{\widehat{A}})), \qquad
\Phi_\Lambda(\mathcal{D}_\Lambda) = C(X_\Lambda),\\
\Phi_\Lambda\circ \rho^{\Lambda,f}_t = \rho^{\widehat{A},f}_t\circ \Phi_\Lambda 
\qquad \text{ for all } f \in C(X_\Lambda, \Z), \quad t \in \T.
\end{gather*}
\end{proposition}
By applying Theorem  \ref{thm:main}, we know the following corollary.
\begin{corollary}\label{cor:5.2}
Let $\Lambda_1$ and $\Lambda_2$ be irreducible nontrivial sofic shifts.
Let $\widehat{A}_1$ and $\widehat{A}_2$ be their transition matrices 
for their left Fischer cover graphs for the sofic shifts 
$\Lambda_1$ and $\Lambda_2$ defined by the formulas \eqref{eq:5.1}, respectively.
Then the right one-sided sofic shifts 
$X_{\Lambda_1}$ and $X_{\Lambda_2}$ are topologically conjugate if and only if
there exists an isomorphism 
$\Phi:\mathcal{O}_{\widehat{A}_1}\longrightarrow \mathcal{O}_{\widehat{A}_2}$
of the simple purely infinite Cuntz--Krieger algebras such that
$\Phi(C(X_{\Lambda_1})) = C(X_{\Lambda_2}))$ and
\begin{equation*}
\Phi\circ \rho^{\widehat{A}_1,f}_t = \rho^{\widehat{A}_2,\Phi(f)}_t \circ\Phi
\qquad \text{ for all }\quad f \in C(X_{\Lambda_1},\Z), \quad t \in \T.
\end{equation*} 
\end{corollary}

\medskip

%%%%%%%%%%%%%%%%%%%%%%%%%%%%%%%%%%%%%%
{\it Acknowledgments:}
This work was  supported by JSPS KAKENHI Grant Number 19K03537.

%%%%%%%%%%%%%%%%%%%%%%%%%%%%%%

%%%%%%%%%%%%%%%%%%%%%%%%%%%%%%%%%%%%%%%%%%%%%%%%%%%%%%%%%%%%%%%%%%%%%%%%%%%%%%

\end{document}